\documentclass[12pt]{amsart}
\usepackage{amsfonts, verbatim}
\usepackage[all,arc]{xy}
\usepackage{enumerate}
\usepackage{mathrsfs}
\usepackage{subcaption}
\usepackage[all]{xy}
\usepackage{amsthm}
\usepackage[T3,OT2,T1]{fontenc}
\usepackage[noenc]{tipa}
\usepackage{comment} 
\usepackage{amssymb}
\usepackage{tikz-cd}
\usepackage{MnSymbol}
\usepackage{cooltooltips}

\usepackage[latin1]{inputenc}
\usepackage{babel}

\usepackage{amsmath}

\usepackage[T1]{fontenc}
\usepackage{textcomp}
\usepackage{bm}

\usepackage{usnomencl}
\addto{\captionsafrikaans}{}
\addto{\captionsenglish}{}

\usepackage{graphicx}
\usepackage{color}
\usepackage{eso-pic}

\usepackage{mathtools}
\usepackage{multicol}
\usepackage{paralist}
\usepackage{geometry}
\usepackage{algorithmic}
\usepackage[ruled,vlined]{algorithm2e}
\usepackage{amsfonts}
\usepackage[linktocpage]{hyperref}
\usepackage{amsthm}
\usepackage{filecontentsdef}
\usepackage{comment}
\usepackage{tikz-cd}
\usepackage{url}
\usepackage{graphicx}
\usetikzlibrary{decorations.pathmorphing}

\DeclareOldFontCommand{\bf}{\normalfont\bfseries}{\mathbf}

\newcommand{\introthmname}{}
\newtheorem{introthminn}{\introthmname}
\newenvironment{introthm}[1]
  {\renewcommand{\introthmname}{#1}\begin{introthminn}}
  {\end{introthminn}}

\newtheorem{introthminn1}{\introthmname}
\newenvironment{introthm1}[1]
  {\renewcommand{\introthmname}{#1}\begin{introthminn1}}
  {\end{introthminn1}}

\newtheorem{theorem}{Theorem}[section]
\newtheorem{corollary}[theorem]{Corollary}
\newtheorem{lemma}[theorem]{Lemma}
\newtheorem*{thm*}{Theorem}
\theoremstyle{definition}
\newtheorem{definition}[theorem]{Definition}
\newtheorem{deflem}[theorem]{Definition-Lemma}
\newtheorem{remark}[theorem]{Remark}

\newcommand\restr[2]{{
\left.\kern-\nulldelimiterspace 
#1 
\vphantom{\big|} 
\right|_{#2} 
}}
\newcommand{\finsub}[0]{{\subseteq}_{\text{\sf{fin}}}}


\title{\bfseries A note on quadratic cyclotomic extensions}
\author{Sophie Marques and Elizabeth Mrema}

\begin{document}




\setcounter{tocdepth}{3}
\maketitle
\begin{center}
\rm e-mail: smarques@sun.ac.za

\it
Department of Mathematical Sciences, 
 Stellenbosch University, \\
Stellenbosch, 7600, 
South Africa\\
\&
NITheCS (National Institute for Theoretical and Computational Sciences), \\
South Africa \\ \bigskip

\rm e-mail: 25138413@sun.ac.za

\it
Department of Mathematical Sciences, 
 Stellenbosch University, \\
Stellenbosch, 7600, 
South Africa\\
\end{center} 
    \tableofcontents

\begin{abstract}
This paper provides two characterizations of the primitive roots of unity in quadratic cyclotomic extensions over arbitrary fields. Firstly, we introduce a mapping from $\mathbb{N}$ to $\mathbb{N}$ crucial for describing these roots, closely tied to their order over the field. Secondly, for any prime $p$, we determine the maximal natural number $n$ such that $\zeta_{p^n}$ defines a quadratic cyclotomic extension over the field $F$. This characterization is uniform across different fields, regardless of their characteristic, and applies to both odd and even primes. \\


\end{abstract}

\section*{Introduction} 
Over the field $\mathbb{Q}$, quadratic cyclotomic extensions do not attract much interest. Indeed, $[\mathbb{Q}(\zeta_n):\mathbb{Q}]=2$ if and only if $n=3$, $n=4$, or $n=6$. We have $\mathbb{Q}(\zeta_3)=\mathbb{Q}(\zeta_6)$. Moreover, $\mathbb{Q}(\zeta_4)$ is not $\mathbb{Q}$-isomorphic to $\mathbb{Q}(\zeta_3)=\mathbb{Q}(\zeta_6)$. Thus, the set of primitive roots of unity defining quadratic cyclotomic extensions is $(\mu_6\cup \mu_4)\setminus \mu_2$ and the set of quadratic cyclotomic extensions consists of the $2$ distinct quadratic cyclotomic extensions $\mathbb{Q}(\zeta_3)$ and $\mathbb{Q}(\zeta_4)$.

Given a finite field $\mathbb{F}_q$ of order $q$, where $q = \wp^n$, $\wp$ is a prime, and $n \in \mathbb{N}$, $\zeta_n$ generates a quadratic cyclotomic extension if and only if $n | (q^2 -1)$ and $n \nmid (q-1)$.

Examining specific examples over finite fields, different cases arise:

\begin{enumerate}
    \item When $q = 13$, $\mathbb{F}_{13^2} = \mathbb{F}_{13}(\zeta_{2^3})$, $\zeta_{2^2} \in \mathbb{F}_{13}$, and the minimal polynomial of $\zeta_{2^3}$ over $\mathbb{F}_{13}$ is $x^2 - \zeta_{2^2}$.
    
    \item When $q = 23$, $\mathbb{F}_{23^2} = \mathbb{F}_{23}(\zeta_{2^i})$ for all $i \in \{2, \dots, 4\}$. The minimal polynomial of $\zeta_4$ over $\mathbb{F}_{23}$ is the quadratic polynomial $x^2 + 1$. Furthermore, the minimal polynomial of $\zeta_{2^3}$ over $\mathbb{F}_{23}$ is $x^2 - (\zeta_{2^3} + \zeta_{2^3}^{-1}) x + 1$. Finally, the minimal polynomial of $\zeta_{2^4}$ over $\mathbb{F}_{23}$ is $x^2 - (\zeta_{2^4} - \zeta_{2^4}^{-1}) x - 1$.
\end{enumerate} 

More broadly, for example, over the extension of rational numbers $\mathbb{Q} (\zeta_{32}+\zeta_{32}^{-1})/\mathbb{Q}$, we obtain 
$$\mathbb{Q} (\zeta_{32}+\zeta_{32}^{-1})(\zeta_{32})=\mathbb{Q} (\zeta_{32}+\zeta_{32}^{-1})(\zeta_{16})=\mathbb{Q} (\zeta_{32}+\zeta_{32}^{-1})(\zeta_{8})= \mathbb{Q} (\zeta_{32}+\zeta_{32}^{-1})(\zeta_{4})$$ 
as a quadratic cyclotomic extension of degree two over $\mathbb{Q} (\zeta_{32}+\zeta_{32}^{-1})$.

In this article, we propose two descriptions of the set of quadratic primitive roots of unity over an arbitrary field, that is, primitive roots of unity defining a quadratic cyclotomic extension of a given field. Each description sheds light on different aspects of the nature of quadratic primitive roots of unity over a given field. While we begin by examining the properties of individual quadratic cyclotomic extensions, broadening our perspective to consider the complete set of quadratic extensions offers insights into their fundamental characteristics.

To be more precise, our first main result is as follows:
\begin{introthm}{Theorem} \label{in1}
The set $\mathscr{M}_{2, \overline{F}}$ of unity roots in $\overline{F}$ which define quadratic cyclotomic extensions can be expressed as:
\[
\mathscr{M}_{2, \overline{F}} = \operatorname{Eq}( \kappa_{F}, 0_{\mu_\infty}) \setminus \mu_{\infty_F};
\]
where:
\begin{itemize}
\item $\mu_{\infty}$ is the set of root of unity and $\mu_{\infty_F}$ is the set of unity roots in $F$;
\item $0_{\mu_{\infty}}: \mu_{\infty} \rightarrow \frac{F(\mu_{\infty})}{F}$ is the zero map sending a root of unity to the additive coset of 0 in $\frac{F(\mu_{\infty})}{F}$; 
\item $\kappa_F: \mu_\infty  \rightarrow \frac{F(\mu_\infty)}{{}_{F}}$ is the map defined as$$\begin{array}{clll}& \zeta_n & \mapsto & \begin{cases}[\zeta_{\operatorname{t}_F(n)} + \zeta_{\operatorname{t}_F(n)}^{-1}]_F \ \ \text{if}\ \operatorname{o}_F(\zeta_{2^{\text{\textepsilon}_n(2)}})\neq 2 \ \text{ and } \ \text{\textepsilon}_n(2)\neq c_2;\\
			 [\zeta_{\operatorname{t}_F(n)} - \zeta_{\operatorname{t}_F(n)}^{-1}]_F\ \ \text{if}\ \operatorname{o}_F(\zeta_{2^{\text{\textepsilon}_n(2)}})\neq 2 \ \text{ and } \ \text{\textepsilon}_n(2)= c_2;\\
			 [\zeta_{2^{ \text{\textepsilon}_n(2)}}(\zeta_{\operatorname{t}_F(n)} - \zeta_{\operatorname{t}_F(n)}^{-1})]_F \ \ \text{ otherwise; }		 
			  \end{cases}\\
	\end{array} $$
	 with
$$\operatorname{t}_F(p^e) =\left\{ \begin{array}{clll}p^e & \text{when }\  p \ \text{ is odd and } \operatorname{o}_F(\zeta_{p^e})\neq1; \\
2^e & \text{when } \ p=2 \ \text{and } \ \operatorname{o}_F(\zeta_{2^e})>2;\\
2 & \text{when} \ p=2 \text{ and } \ \operatorname{o}_F(\zeta_{2^e})=2;\\
1 & \text{when} \ \operatorname{o}_F(\zeta_{p^e})=1.
\end{array} \right.$$ 
for any $p\in \mathbb{P}$ and $e\in \mathbb{N}$, $\operatorname{t}_F(n)=\prod \limits_{\substack{p\in \mathbb{P}, p|n}} \operatorname{t}_F(p^{\text{\textepsilon}_n(p)})$, and $\text{\textepsilon}_n(p)$ is the maximal power of $p$ dividing $n$, for any $n\in \mathbb{N}$;
	\item $\operatorname{Eq}( \kappa_{F}, 0_{\mu_\infty})= \{ \zeta \in \mu_\infty | \kappa_{F}(\zeta)=[0]_F\}.$
\end{itemize}
with $c_2$ the unique integer $e\in \mathbb{N}$ such that $\zeta_{2^e}\notin F$, $\operatorname{t}_F(2^e)\neq 2$ and $\zeta_{2^e}-\zeta_{2^e}^{-1}\in F$ and $\frac{F(\mu_\infty)}{F}$ being the quotient set when both $F(\mu_\infty)$ and $F$ are viewed as additive groups, and $[a]_F$ is the coset of $a$ in this quotient.
\end{introthm}
This initial characterization underscores the importance of the order of a primitive root of unity and the natural number $c_2$, as well as the map $\kappa_F$ via the map $\operatorname{t}_F$, to grasp the structure of quadratic cyclotomic extensions. The map $\operatorname{t}_F$ covers all possible cases in a single function, a definition made feasible once Theorem \ref{theorem-valueofk} is established. A key result in establishing this theorem was to show that $2$ power primitive roots of unity defining quadratic cyclotomic extensions have an order of either $2$ or $2^{e-1}$ over $F$ (see Theorem \ref{order-zeta2e}). The second main result is as follows:
\begin{introthm}{Theorem}  
\label{in2}
	\begin{enumerate}
		\item The set $\mathscr{M}_{2, \overline{F}}$ of primitive roots of unity in $\mu_{\infty}$ defining quadratic extensions can be described as a disjoint union of group differences: $$\mathscr{M}_{2, \overline{F}}=\coprod_{M\in \mathcal{S}_{F,\operatorname{max}}}(\bigvee_{p\in M}\mu_{p^{\nu_{p^\infty_F}}}\bigvee_{p \in \mathbb{P}\setminus M}\mu_{p^{\ell_{p_F^\infty}}}\setminus \bigvee_{p \in \mathbb{P}}\mu_{p^{\ell_{p_F^\infty}}});$$
		\item The set $\mathscr{C}_{2,\overline{F}}$ of quadratic cyclotomic extensions is in bijection with $\mathcal{S}_{F,\operatorname{max}};$	
		\end{enumerate}
	where 
\begin{itemize}
			\item $\mathbb{P}$ is the set of prime numbers;
\item $\mu_n$ is the set of $n^{th}$ roots of unity, where $n \in \mathbb{N}$;
    \item $\mathcal{S}_{F,\operatorname{max}}$ is the set of upper bounds of maximal chains in $\mathcal{S}_F$, \\
    with $\mathcal{S}_F:=\{S\in P(\mathbb{P})| 
\forall \ p\in S, \exists e_p \in \mathbb{N}:[\zeta_{p^{e_p}}\in \mathscr{M}_{2, \overline{F}}\ \wedge \forall \ B\finsub S, \zeta_{\prod\limits_{p\in B}p^{e_p}} \in \mathscr{M}_{2, \overline{F}}]\}.$
\item $\nu_{p^\infty_F}= \begin{cases} \nu_{p^\infty_F}^++1 & \text{when} \  p=2 \ \text{and} \ F \  \text{has property} \ \mathcal{C}_2 ;\\  \nu_{p^\infty_F}^+ & \text{otherwise},\end{cases} $ \\ 
with $\nu_{p^\infty_F}^+=\begin{cases}
			\operatorname{max}\left\{ k \in \mathbb{N} \cup \{ 0 \} |\zeta_{\operatorname{t}_F(p^k)}+\zeta_{\operatorname{t}_F(p^k)}^{-1}\in F \right\} & \text{if it exists}; \\
			\infty & \text{otherwise};
			\end{cases}
			$\\
and $F$ has property $\mathcal{C}_2$ if there exists $e\in \mathbb{N}$ $\zeta_{2^e}\notin F$, $\operatorname{t}_F(2^e)\neq 2$ and $\zeta_{2^e}-\zeta_{2^e}^{-1}\in F$;    
\item $\ell_{p^\infty_F}=\begin{cases}
				\operatorname{max}\{k \in \mathbb{N} \cup \{ 0 \} |\zeta_{p^k}\in F\} & \text{if it exists}; \\
				\infty & \text{otherwise};
			\end{cases}
			$ 
\item $A \coprod B$ denotes the coproduct of some sets $A$ and $B$ and $A \bigvee B$ denote the join of two subgroups $A$ and $B$ of a group $C$; above the infinite join is simply the set of finite products of elements taken in the family of groups considered.
\end{itemize}
\end{introthm}
This theorem reveals two natural numbers $\ell_{p^\infty_F}$ and $\nu_{p^\infty_F}$ depending only on the base field $F$ and a fixed prime $p$. It is important to note that the latter natural number depends on the property $\mathcal{C}_2$, which we demonstrate to be a significant property of the base field when studying quadratic cyclotomic extensions. We fix a prime number $p$. The natural number $\ell_{p^\infty_F}$ is very natural to define as it represents the maximal natural number $n$ such that $\zeta_{p^n}\in F$. As for $\nu_{p^\infty_F}$, it represents, when there exists a quadratic cyclotomic $p$ power primitive root of unity, the maximal natural number $n$ such that $\zeta_{p^n}$ is quadratic. This integer is uniformly defined regardless of whether $p$ is odd or even and over any arbitrary field, regardless of its characteristic. Finally, $\mathcal{S}_{F,\operatorname{max}}$ permits to bring together the product of those quadratic cyclotomic power roots of unity that remain quadratic.
We also note that Theorem \ref{nu2f} establishes an interesting connection between the integers $\ell_{p^\infty_F}$ and $\nu_{p^\infty_F}$. 

In \cite[Theorem 3.2.16]{thesis} and \cite[Theorem 3.9]{2-power}, we successfully characterize the cyclicity of $2$ power cyclotomic extensions using the results of this article with a property defined only on the base field involving the natural numbers $\nu_{p^\infty_F}$.
The construction of $\nu_{p^\infty_F}$ build upon the integers $\nu_{p^\infty_F}^+$ that appears in many works on cyclotomic extensions without characterization. For instance, in \cite{subfieldsradicalextension}, the authors use the natural numbers $\nu_{p^\infty_F}^+$ to describe the subfields of radical extensions and to determine normal radical extensions and their Galois groups. In \cite{several}, the natural numbers $\nu_{p^\infty_F}^+$ were used in proving Schinzel's theorem. The work of Schinzel, in \cite[Theorem 3]{schinzel} and \cite[Theorem 2.1]{several}, in particular, demonstrated that in the non-Kummer case, a new type of isomorphisms between two simple radical extensions emerges. Although the precise conditions for the emergence of this new type of isomorphisms was not shown in his theorem, the conditions described in \cite[Theorem 2.1]{several} involve $\ell_{2^\infty_F}$, $\nu_{2^\infty_F}$, and the property $\mathcal{C}_2$. In \cite[Theorem 5.2.6]{thesis}, we managed to characterize, based on a property of the base field involving the integers $\nu_{p^\infty_F}$, when the new type of isomorphisms emerging in Schinzel's work appears between simple radical extensions. This result connects the properties of quadratic primitive roots of unity to broader questions about general cyclotomic extensions and simple radical extensions.

In the first section, we introduce two results on general cyclotomic extensions that we will use throughout the paper. The first result, Lemma \ref{equality}, characterizes when two cyclotomic extensions are equal if they have equal degrees. The second result, Theorem \ref{ordergaloiscyclotomic}, provides some constraints on the cardinality of the automorphism group of a general cyclotomic field. We end the section with a corollary of the Waring formula (see Lemma \ref{waring}).

In the second section, after identifying the general form for the minimal polynomial of a primitive root of unity with quadratic minimal polynomial, we identify the possible orders for a primitive root of unity defining a quadratic cyclotomic extension (see Lemma \ref{remark1} and Theorem \ref{order-zeta2e}). This culminates in Theorem \ref{theorem-valueofk}, which consolidates most of our results from this section in one place and expresses the minimal polynomial of a primitive $n^{th}$ root of unity with quadratic minimal polynomial in terms of $\operatorname{o}_F(\zeta_n)$ and $\operatorname{d}_F(n)$.

In the third section, we prove the two main theorems stated above.

We conclude the paper by characterizing the sets of general quadratic extensions as quotients by a group action, thereby revealing the essence of quadratic extension. Furthermore, we introduce embeddings of sets of quadratic cyclotomic extensions into the set of general quadratic extensions (see Definitions \ref{chi-rad} and \ref{chi-rad2}), illustrating how the sets of quadratic cyclotomic extensions are integrated within the larger framework of all quadratic extensions. We include the study of the inseparable case for completeness.
\newpage
\section*{Notation and symbols}
In this paper, 
 \begin{itemize} 
 \item $F$ is a field with characteristic $\wp$. 
\item $\overline{F}$ is one fixed algebraic closure of the field $F$.
 \item $p$ is a prime number unless explicitly stated otherwise.
 \item $n$ is a natural number unless explicitly stated otherwise.
 \end{itemize}
For the sake of simplicity throughout this paper, we will assume that all the field extensions discussed are subfields of the above chosen algebraic closure $\overline{F}$.
In the next table, we introduce some notation/symbols that will be used in this paper. Some of these symbols are classical and widely known, but we introduce them here to avoid any possible alternative meaning. In this paper, we will also often introduce symbols using definitions.
\begin{table}[ht]
	\raggedright 
	\renewcommand{\arraystretch}{0.9}
	\textbf{Set theory}
	\begin{tabular}{cp{14cm}c}
		\multicolumn{1}{c}{} & \multicolumn{1}{c}{} & \multicolumn{1}{c}{} \\
		$\mathbb{N}$ & The set of natural numbers (excluding $0$). & \\
		$\mathbb{P}$ & The set of prime numbers.  & \\
		$P(\mathbb{P})$ & The power set of a set $\mathbb{P}$. & \\
		$A \finsub B$ &  $A$ is a finite subset of a set $B$. & \\
			$[j]_{n}$ & The equivalence class of $j\in \mathbb{Z}$ modulo $n$. We shall simply use the notation $[j]$ when $n$ is clear from the context. &  \\
 $\text{\textepsilon}_n(p)$  & The maximal power of a prime number $p$ dividing $n$. In particular, when $p\nmid n$, $\text{\textepsilon}_n(p)=0$. We also write $p^{\text{\textepsilon}_n(p)} || n$. & \\
$\operatorname{q}_n(p)$ & The quotient of the Euclidean division of $n$ by $p^{\text{\textepsilon}_n (p)}$. &  \\
$A^*$ & $A\setminus\{0\}$, for all $A \subseteq F$ containing $0$.&\\
$A \coprod B$ & The coproduct of some sets $A$ and $B$.&\\
 $A \bigvee B$ & The join of two subgroups $A$ and $B$ of a group $C$.&

		\end{tabular}
	\end{table}
			\begin{table}[ht]
			\raggedright 
			\renewcommand{\arraystretch}{0.9}
		\textbf{Field extensions}
			
			\begin{tabular}{cp{14cm}c}
				\multicolumn{1}{c}{} & \multicolumn{1}{c}{} & \multicolumn{1}{c}{} \\
			$ \cong_F$ & An $F$-isomorpism between field extensions of $F$.& \\
		$\operatorname{min}(\alpha, F)$ & The minimal polynomial of an element $\alpha \in \overline{F}$ over $F$. &\\
		$\operatorname{o}_F(\alpha)$ & The order of an element $\alpha\in \overline{F}$ over $F$.& \\
		
			\end{tabular}
	\end{table}

\begin{table}[ht]
\raggedright 
\renewcommand{\arraystretch}{0.9}
\textbf{Roots of unity and cyclotomic extension}
\begin{tabular}{cp{14cm}c}
\multicolumn{1}{c}{} & \multicolumn{1}{c}{} & \multicolumn{1}{c}{} \\
$\text{\textctc}_{F}(n)$  &  The maximum power of $\wp$ dividing $n$ when $\wp>0$ or $0$ otherwise. & \\
		$\zeta_n$ & A primitive $n^{th}$ root of unity. Everytime the notation $\zeta_n$ is used, we implicitly assume that $\text{\textctc}_{F}(n)=0$. When $n$ is a power of a prime number $p$, we say that $\zeta_n$ is a $p$ power primitive root of unity. When $[F(\zeta_n)/F]=2$, we say that $\zeta_n$ is a quadratic primitive root of unity. \\
		$\mathcal{P}_n$ &  The set of the primitive $n^{th}$ roots of unity. We note that all elements in $\mathcal{P}_n$ have the same order over $F$. &\\
		$\mu_n$ &  The set of the $n^{th}$ roots of unity .& \\
			${j}_{\sigma,n}$ & The integer such that $\sigma(\zeta_n) = \zeta_n^{{j}_{\sigma,n}}$ where $\sigma$ is an automorphism in the Galois group of the extension $F(\zeta_n)/F$. We may write ${j}_{\sigma}$, when $n$ is clear from context.&  \\
				$\sigma_k$ &  The $F$-automorphism from $F(\zeta_n)$ to $F(\zeta_n)$ such that $\sigma_k(\zeta_n)=\zeta_n^k$ where $k\in \mathbb{N}$. & \\
		$\operatorname{d}_F(n)$ &  The maximum number dividing $n$ such that $\zeta_{\operatorname{d}_F(n)}\in F$. We note that $\operatorname{d}_F(n) = \frac{n}{\operatorname{o}_F(\zeta_n)}$ and  $ \mu_{\operatorname{d}_F(n)}= \mu_n \cap F$. &
		\end{tabular}
	\end{table}

\normalsize 

\newpage 
\section{Preliminary material and notation}
The following criterion gives a characterization for two cyclotomic extensions to be equal when they have equal degrees, useful in proving Lemma \ref{powerset-cyclotomic}.
\begin{lemma}\label{equality}
Let $n, m, l \in \mathbb{N}$ and $[F(\zeta_n):F]=[F(\zeta_m):F]=l$. Then the following assertions are equivalent:
\begin{enumerate}
    \item $F(\zeta_n)=F(\zeta_m)$.
    \item $[F(\zeta_{\operatorname{lcm}(n,m)}):F]=l$.
\end{enumerate}
\end{lemma}
\begin{proof}
The equivalence of this lemma can be deduced after establishing the following results: 
\begin{enumerate} 
    \item For any $n , m \in \mathbb{N}$,
    \[\zeta_n \zeta_m = \prod_{\substack{ p| nm \\ p \text{ prime }}} \zeta_{p^{\max (\text{\textepsilon}_n(p), \text{\textepsilon}_m(p))}}^{p^{|\text{\textepsilon}_n(p)-\text{\textepsilon}_m(p)|}+1}.\] 
    Moreover, 
    \begin{itemize} 
        \item when $\text{\textepsilon}_n(2)\neq \text{\textepsilon}_m(2)$ or $\text{\textepsilon}_n(2)= \text{\textepsilon}_m(2)=0$, $\zeta_n \zeta_m$ is a primitive $\operatorname{lcm}(n,m)^{th}$ root of unity, and 
        \item when $\text{\textepsilon}_n(2)= \text{\textepsilon}_m(2)$ non-zero, $\zeta_n \zeta_m$ is a primitive $\frac{\operatorname{lcm}(n,m)}{2}^{th}$ root of unity.
    \end{itemize}
    \item Given $n, m \in \mathbb{N}$, we have 
    \begin{enumerate} 
        \item $F(\zeta_n\zeta_m)=F(\zeta_{\operatorname{lcm}(n,m)})$ when $\text{\textepsilon}_n(2)\neq \text{\textepsilon}_m(2)$ or $\text{\textepsilon}_n(2)=\text{\textepsilon}_m(2)=0$. 
        \item $F(\zeta_n \zeta_m)=F(\zeta_{\frac{\operatorname{lcm}(n,m)}{2} })$ when $\text{\textepsilon}_n(2)= \text{\textepsilon}_m(2)\neq0$. 
    \end{enumerate} 
    For more details, one can refer to \cite[Lemma 1.2.38]{thesis} and \cite[Corollary 1.2.39]{thesis}.
\end{enumerate} 
\end{proof}

The second result of this section gives a more precise constraint than \cite[Theorem 2.3]{conradcyclotomic} on the order of the automorphism of a cyclotomic extension and will permit us to determine the order of  $2$ power quadratic primitive roots of unity in Lemma \ref{order-zeta2e}.
\begin{theorem}\label{ordergaloiscyclotomic}
	Let $n\in \mathbb{N}$.
We have $\operatorname{Gal} ( F(\zeta_n)/ F)$ is isomorphic to a subgroup of $\operatorname{Aut}_{\mu_{\operatorname{d}_F(n)}} (\mu_n) $.
In particular, $\operatorname{o}(\operatorname{Gal}(F(\zeta_n)/F))|\frac{\phi(n)}{\phi(\operatorname{d}_F(n))}$ where $\phi$ is Euler's function.
\end{theorem}

\begin{proof}		
One can define a map $\Psi: \operatorname{Gal}(F(\zeta_n)/F) \rightarrow \operatorname{Aut}_{\mu_{\operatorname{d}_F(n)}}(\mu_n)$ sending $\sigma$ to $\sigma|_{\mu_n}$, obtained by corestricting and restricting $\sigma$ to $\mu_n$, since $\zeta_{\operatorname{d}_F(n)} \in F$. 
Moreover, one can prove that $\Psi$ is a group isomomorphism. By \cite[Lemma 1.2.42]{thesis}, we have $\operatorname{o}\left(\operatorname{Aut}_{\mu_{\operatorname{d}_F(n)}}(\mu_n)\right) = \frac{\operatorname{o}(U_n)}{\operatorname{o}(U_{\operatorname{d}_F(n)})}.$ Hence $\operatorname{o}(\operatorname{Gal}(F(\zeta_n)/F))|\frac{\phi(n)}{\phi(\operatorname{d}_F(n))}.$
\end{proof}
From the previous theorem we deduce the following condition about the degree of cyclotomic extension $F(\zeta_{p^e})/F$ when $p$ is an odd prime number. 
\begin{corollary}\label{degree-odd}
	Let $p$ odd, $\zeta_{p^e} \notin F$ and $\zeta_p \in F$, then $[F(\zeta_{p^e}):F]$ is a $p$ power.
\end{corollary}
\begin{proof}
From the assumption, $\operatorname{d}_F(p^e)= p^f$ where $f\geq1$. 
	The result follows immediately  by Theorem \ref{ordergaloiscyclotomic} above because $\operatorname{o}(\operatorname{Gal}(F(\zeta_{p^e})/F))|\frac{\phi(p^e)}{\phi(\operatorname{d}_F(p^e))}=p^{e-f}$. 
\end{proof}

We end the section with a direct consequence of the Waring formula, yet very useful in this paper.
\begin{lemma}\label{waring}
Given $x, y \in \overline{F}$ and $ t\in \mathbb{N}$ such that $x +y \in F$ and $xy\in F$. Then, $x^{t}+y^{t}\in F$ for all $t\in \mathbb{N}$.  
\end{lemma}
\begin{proof} By Waring formula (see \cite[\S 4.9]{comtet}) we have $$x^{t}+y^{t}=\sum \limits^{ \lfloor \frac{t}{2} \rfloor }_{i=0}(-1)^i\frac{t}{t-i}\binom{t-i}{i}\left(x+y\right)^{t-2i}(xy)^i $$ 
	$ \text{where} \ \binom{t-i}{i}\ \text{is a binomial coefficient}$ and $ \lfloor \frac{t}{2} \rfloor$ is the floor function at $  \frac{t}{2}$.
	Since $x+y\in F$, $xy\in F$ and   $(-1)^i\frac{t}{t-i}\binom{t-i}{i}\in \mathbb{Z}$ then  $x^t+y^t\in F$. 
\end{proof}

\section{Properties of quadratic cyclotomic extensions}
Our objective is to examine the properties of quadratic cyclotomic extensions. We begin by noting that quadratic cyclotomic extensions are always separable. Indeed, when $\wp=2$, $n$ cannot be even because there are no even primitive roots of unity in characteristic $2$.
\subsection{General form of the minimal polynomial of quadratic roots of unity}
 We will now establish the general form of the minimal polynomial of a primitive root of unity defining a cyclotomic extension of degree $2$.
\begin{lemma}\label{minimalpolynomialzetan}
 Let $n \in \mathbb{N}$ and suppose that $[F(\zeta_n):F]=2$. The minimal polynomial of $\zeta_n$ is of the form $x^2-(\zeta_n+\zeta_n^k)x+\zeta_n^{k+1}$ for some unique $k\in \{1, \cdots, n-1\}$ and $(k,n)=1$. Moreover,  $\operatorname{o}_F(\zeta_n)|k^2-1$.
\end{lemma}
\begin{proof}
Let  $\sigma \in \operatorname{Gal}(F(\zeta_n)/F)$ with $\sigma\neq \operatorname{Id}$. The description of the minimal polynomial follows directly from the fact that $\sigma(\zeta_n)= \zeta_n^k$ where $k\in \{1, \cdots, n-1\}$ with $(k, n)=1$. To prove the uniqueness of $k$, we may use contradiction. We suppose that there is $k'\in \{1, \cdots, n-1\}$ such that $k\not \equiv k' \mod n$, $\zeta_n+\zeta_n^{k'}\in F$ and $\zeta_n^{k'+1}\in F$. So that $\zeta_n(\zeta_n^k-\zeta_n^{k'})=\zeta_n^{k+1}-\zeta_n^{k'+1}$ which implies that $\zeta_n \in F$ since $\zeta_n^k-\zeta_n^{k'}\in F$ and $\zeta_n^{k+1}-\zeta_n^{k'+1}\in F$. This is a contradiction with the initial assumption. Hence $k$ is unique. Further, $\sigma^2(\zeta_n)=\zeta_n$. Hence $\zeta_n^{k^2-1}=1$. Hence $\operatorname{o}_F(\zeta_n)|(k^2-1)$.
\end{proof}
The uniqueness of $k$ permits us to establish the next definition. 
\begin{definition} 
 Let $n\in \mathbb{N}$ such that $[F(\zeta_n):F]=2$. We denote $\text{ \textlyoghlig}_n$ as the only integer in $\{1, \cdots, n-1\}$ coprime with $n$ such that $\zeta_n + \zeta_n^{\text{ \textlyoghlig}_n} \in F$ and $\zeta_n^{\text{ \textlyoghlig}_n+1}\in F $. 
\end{definition}
When $\wp \neq 2$, we know that any quadratic extension is a simple radical extension, meaning it admits a radical generator whose minimal polynomial is of the form $x^2 - a$, where $a$ is in the base field. When $\wp= 2$, we know that any separable quadratic extension is an Artin-Schreier extension, meaning it admits an Artin-Schreier generator whose minimal polynomial is of the form $x^2 - x - a$, where $a$ is in the base field. We explicitly provide a radical generator in the case $\wp \neq 2$ and an Artin-Schreier generator in the case $\wp= 2$ for quadratic cyclotomic extensions.
\begin{corollary} \label{generator} 
Let $n \in \mathbb{N}$ and suppose that $[F(\zeta_n):F]=2$.
Then 
\begin{enumerate} 
\item $\zeta_n-\zeta_n^{\text{ \textlyoghlig}_n}$ is a radical generator for $F(\zeta_n)$ over $F$, when $\wp \neq 2$;
\item $\frac{\zeta_n}{\zeta_n+\zeta_n^{\text{ \textlyoghlig}_n}}$ is an Artin-Scheier generator for $F(\zeta_n)$ over $F$, when $\wp= 2$. 
\end{enumerate} 
\end{corollary}
\begin{proof}
Suppose that $[F(\zeta_n):F]=2$.
	\begin{enumerate} 
	\item It is enough to prove that $\zeta_n-\zeta_n^{\text{ \textlyoghlig}_n} \notin F$ and that the $\zeta_n-\zeta_n^{\text{ \textlyoghlig}_n}$ is a root of a degree $2$ a radical polynomial over $F$. By Lemma \ref{minimalpolynomialzetan} and Lemma \ref{waring}, we have that $(\zeta_n-\zeta_n^{\text{ \textlyoghlig}_n})^2=\zeta_n^2+\zeta_n^{2{\text{ \textlyoghlig}_n}}-2\zeta_n^{\text{ \textlyoghlig}_n+1}\in F$. We argue by contradiction to prove that $\zeta_n-\zeta_n^{\text{ \textlyoghlig}_n} \notin F$. Suppose that $\zeta_n-\zeta_n^{\text{ \textlyoghlig}_n} \in F$. Since $\sigma(\zeta_n-\zeta_n^{\text{ \textlyoghlig}_n})=\zeta_n^{\text{ \textlyoghlig}_n}-\zeta_n$ and  $\zeta_n-\zeta_n^{\text{ \textlyoghlig}_n}\in F$ then $\zeta_n-\zeta_n^{\text{ \textlyoghlig}_n}=\sigma(\zeta_n-\zeta_n^{\text{ \textlyoghlig}_n})=\zeta_n^{\text{ \textlyoghlig}_n} -\zeta_n$ implies $\zeta_n=\zeta_n^{\text{ \textlyoghlig}_n}= \sigma (\zeta_n)$ implies $\zeta_n \in F$. This is a contradiction since $[F(\zeta_n):F]=2$.  
	\item When $\wp = 2$. Clearly, $\frac{\zeta_n}{\zeta_n+\zeta_n^{\text{ \textlyoghlig}_n}}\notin F$ and it is a root of the Artin-Schreier polynomial $x^2 -x+\frac{\zeta_n^{\text{ \textlyoghlig}_n+1}}{(\zeta_n+\zeta_n^{\text{ \textlyoghlig}_n})^2}$ over $F$. 
	\end{enumerate}
\end{proof}
\subsection{The order of roots of unity defining a quadratic extension}

The following lemma proves that the order of an odd primitive $p$ power primitive root of unity generating a quadratic cyclotomic extension is $p^e$.
\begin{lemma} \label{remark1}
		 If $p$ is an odd prime number and $[F(\zeta_{p^e}):F]=2$, then $\operatorname{d}_F(p^e)=1$. In particular, $\operatorname{o}_F (\zeta_{p^e} )= p^e$. 
\end{lemma} 

\begin{proof} 
This follows directly from Corollary \ref{degree-odd}.
\end{proof}

In our effort to understand quadratic cyclotomic extensions over the field $F$ generated by $2$ power primitive root of unity, we also obtain that their order can only be either $2$ or $2^{e-1}$, this result is key in the discussion to follow. 
\begin{theorem}\label{order-zeta2e}
	Let $e > 1$.  If $[F(\zeta_{2^e}):F]=2$, then $\operatorname{o}_F(\zeta_{2^e})\in \{ 2, 2^{e-1}\}$.
\end{theorem}
\begin{proof}
	Let $[F(\zeta_{2^e}):F]=2$, $\operatorname{o}_F(\zeta_{2^e})=2^t$ and $\sigma$ be a non trivial automorphism in $Gal (F(\zeta_{2^e})/F)$. We have $t\geq 1$, since $\zeta_{2^e}\notin F$. 
	By definition of the order, we have $\zeta_{2^e}^{2^{t-1}}\notin F$ and $\zeta_{2^e}^{2^{t}}\in F$.
	So, $[F(\zeta_{2^e}^{2^{t-1}}):F]=2$ and $x^2-\zeta_{2^e}^{2^t}$ is the minimal polynomial of $\zeta_{2^e}^{2^{t-1}}$ over $F$. Therefore, $F(\zeta_{2^e}^{2^{t-1}})=F(\zeta_{2^e})$ and $\sigma(\zeta_{2^e}^{2^{t-1}})=-\zeta_{2^e}^{2^{t-1}}=\zeta_{2^e}^{2^{e-1}+2^{t-1}}$. We know that $\sigma(\zeta_{2^e})=\zeta_{2^e}^{\text{ \textlyoghlig}_{2^e}}$ and $\text{ \textlyoghlig}_{2^e}\in \{1,\cdots, 2^e-1\}$ with $(\text{ \textlyoghlig}_{2^e}, 2^e)=1$. 
	Hence,  $\sigma(\zeta_{2^e}^{2^{t-1}})=\sigma(\zeta_{2^e})^{2^{t-1}}=\zeta_{2^e}^{2^{t-1}\text{ \textlyoghlig}_{2^e}}$. It follows that $\zeta_{2^e}^{2^{e-1}+2^{t-1}}= \zeta_{2^e}^{2^{t-1}\text{ \textlyoghlig}_{2^e}}$. Therefore, $\text{\textlyoghlig}_{2^e}2^{t-1}\equiv 2^{e-1}+2^{t-1}\mod 2^e$  and so $\text{ \textlyoghlig}_{2^e}\equiv 2^{e-t}+1 \mod 2^{e-t+1}$. This implies that $\text{\textlyoghlig}_{2^e}+1= 2+2^{e-t}+2^{e-t+1}r$ for some $r\in \mathbb{Z}$. Moreover, we know by Lemma \ref{minimalpolynomialzetan},  $\operatorname{o}_F( \zeta_{2^e})= 2^t|\text{ \textlyoghlig}_{2^e}+1$. As a result, $2+2^{e-t}+2^{e-t+1}r\equiv 0 \mod 2^{t}$.  This implies $2^{e-t-1}+2^{e-t}r+1\equiv 0 \mod 2^{t-1}$ which implies $2^{e-t-1}(-1-2r)\equiv 1 \mod 2^{t-1}$. As a consequence, $2^{e-t-1}$ has a multiplicative inverse modulo $2^{t-1}$. But, when $t\neq 1$, $2^{e-t-1}$ and $2^{t-1}$ are not coprime if and only if $e-t-1=0$. That is $t=e-1$. As a conclusion, we have either $t=1$ or $t = e-1$. Hence, we have proven the result.       
\end{proof}

The next corollary computes $(\operatorname{d}_F(n) , \operatorname{o}_F (\zeta_n) )$ when $\zeta_n$ defines a quadratic cyclotomic extension. 
\begin{corollary} \label{coprime}
Let $n\in \mathbb{N}$. If $[F(\zeta_n):F]=2$, then we have 
\begin{enumerate} 
\item $(\operatorname{d}_F(n) , \operatorname{o}_F (\zeta_n) ) = 1,$ when $\zeta_{2^{\text{\textepsilon}_2(n)}}\in F$;
\item $(\operatorname{d}_F(n) , \operatorname{o}_F (\zeta_n) ) = 2,$ otherwise
\end{enumerate}
\end{corollary}

\begin{proof} 
\begin{enumerate} 
\item We assume, for the sake of contradiction, that there exists $p$ prime number such that $p | (\operatorname{d}_F(n) , \operatorname{o}_F(\zeta_n))$. Then $p$ is an odd prime number, since $\zeta_{2^{\text{\textepsilon}_2(n)}}\in F$. Since $p | \operatorname{o}_F( \zeta_n)$, it follows that $p^{\text{\textepsilon}_n (p)} =\operatorname{o}_F(\zeta_{p^{\text{\textepsilon}_n (p)}}) \parallel \operatorname{o}_F (\zeta_n )  $, as established in Lemma \ref{remark1}. However, then $p$ cannot divide $\operatorname{d}_F(n)$, since $\operatorname{d}_F(n) = \frac{{n}}{\operatorname{o}_F( \zeta_n)}$.
\item When $\text{\textepsilon}_n (2) > 1$. We have  $ (\operatorname{d}_F(n), \operatorname{o}_F(\zeta_n))=2$ by Lemma \ref{order-zeta2e} and (1). 
\end{enumerate}
\end{proof}
The following result is a direct consequence of Corollary \ref{coprime}
\begin{corollary}
Let $n \in \mathbb{N}$. If $[F(\zeta_n):F]=2$, we have 
 \begin{enumerate}
\item $\mu_{\operatorname{d}_F(n)}\cap \mu_{\operatorname{o}_F(\zeta_n)}=\{1\}$ when $\zeta_{2^{\text{\textepsilon}_n(2)}}\in F$;  
\item $\mu_{\operatorname{d}_F(n)}\cap \mu_{\operatorname{o}_F(\zeta_n)}=\mu_2$, otherwise.
\end{enumerate}
\end{corollary}

\subsection{Properties of a cyclotomic extension of degree 2}

In the following lemma, we describe quadratic a cyclotomic extension generated by a $p$ power primitive root of unity when $p$ is an odd prime.

\begin{lemma}\label{allcyclotomicofdegree2}
Let $n \in \mathbb{N}$ and $[F(\zeta_n):F]=2$. Let $p$ be an odd prime number dividing $n$ such that $p|\operatorname{o}_F(\zeta_n)$. Then:
\begin{enumerate}
	\item $F(\zeta_n ) = F(\zeta_{p^t})$ for all $t \in \{ 1, \cdots , \text{\textepsilon}_n (p)\}$;
	\item $p^{\text{\textepsilon}_n (p)} \parallel \operatorname{o}_F( \zeta_n)$ and $\operatorname{o}_F(\zeta_{p^t})=p^t$ for all $t\in \{1, \cdots, \text{\textepsilon}_n (p) \}$;
	\item The minimal polynomial of $\zeta_{p^t}$ over $F$ is $x^2-(\zeta_{p^t}+\zeta_{p^t}^{-1})x+1$ for all $t \in \{ 1, \cdots, \text{\textepsilon}_n (p) \}.$
\end{enumerate} 
\end{lemma} 
\begin{proof}
\begin{enumerate}
	\item For all $t\in \{1, \cdots, \text{\textepsilon}_n (p) \}$, $F(\zeta_{p})\subseteq F(\zeta_{p^t}) \subseteq F(\zeta_n)$. Since $p|\operatorname{o}_F(\zeta_n)$, $\zeta_p\in F(\zeta_n)\setminus F$, $[F(\zeta_p):F]=2$. Therefore, $F(\zeta_{p^t})=F(\zeta_n)$, for all $t\in \{1, \cdots, \text{\textepsilon}_n (p)\}$.
		
	\item  This follows directly from $(1)$ and Lemma \ref{remark1}.
		\item 
		Since $[F(\zeta_{p^t}):F]=2$, for all $t\in \{1, \cdots, \text{\textepsilon}_n (p)\}$, then, by Lemma \ref{minimalpolynomialzetan}, we have that the minimal polynomial of $\zeta_{p^t}$ is $x^2-(\zeta_{p^t}+\zeta_{p^t}^{\text{ \textlyoghlig}_{p^t}})x+\zeta_{p^{t}}^{\text{ \textlyoghlig}_{p^t}+1}$ where $\text{ \textlyoghlig}_{p^t}\in \{1, \cdots, p^t-1\}$ with $(\text{ \textlyoghlig}_{p^t}, p^t)=1$ and $\operatorname{o}_F(\zeta_{p^{t}})|(\text{ \textlyoghlig}_{p^t}+1)$. By (2) above we have  $\operatorname{o}_F(\zeta_{p^{t}})=p^t$ which implies that $\text{ \textlyoghlig}_{p^t} \equiv -1 \mod p^t$. Therefore, the minimal polynomial of $F(\zeta_{p^t})$ is $x^2-(\zeta_{p^t}+\zeta_{p^t}^{-1})x+1$ for all $t\in \{1, \cdots, \text{\textepsilon}_n (p)\}$ as desired. 
	
	\end{enumerate} 
\end{proof}

In the next result, we provide a description of the possible minimal polynomials of a primitive $2$ power primitive root of unity generating a quadratic cyclotomic extension over $F$.
\begin{lemma}\label{valuesofkfor2^e}
	Let $e\in \mathbb{N}$ and suppose that $[F(\zeta_{2^e}):F]=2$. Then one of the following assertion is satisfied.
	
	\begin{enumerate}
		\item $\text{ \textlyoghlig}_{2^e}\equiv 1+2^{e-1} \mod 2^e$ and $\operatorname{o}_F ( \zeta_{2^e})=2$. In particular, $\operatorname{min}(\zeta_{2^e}, F)=x^2-\zeta_{2^e}^2$. 
		\item $F(\zeta_{2^e})=F(\zeta_{2^e}^{2^j})$ for all $j\in \{1, \cdots, e-2\}$, 
		$\operatorname{o}_F ( \zeta_{2^e})=2^{e-1}$, and either 
		\begin{enumerate}
			\item $\text{ \textlyoghlig}_{2^e} \equiv 2^{e-1}-1\mod 2^e$. In particular, $\operatorname{min}(\zeta_{2^e}, F)=x^2-(\zeta_{2^e}-\zeta_{2^e}^{-1})x-1$ and $\zeta_{2^e}-\zeta_{2^e}^{-1}\in F$, 
			or
			\item $\text{ \textlyoghlig}_{2^e} \equiv -1\mod 2^e$. In particular, $\operatorname{min}(\zeta_{2^e}, F)=x^2-(\zeta_{2^e}+\zeta_{2^e}^{-1})x+1$ and $\zeta_{2^e}+\zeta_{2^e}^{-1}\in F$. 
		\end{enumerate}
	\end{enumerate}
\end{lemma}

\begin{proof}
	Suppose that $[F(\zeta_{2^e}):F]=2$. Since $\zeta_{2} \in F$, we have $e>1$. Then by Lemma \ref{minimalpolynomialzetan} $\operatorname{min}(\zeta_{2^e}, F)=x^2-(\zeta_{2^e}+\zeta_{2^e}^{\text{ \textlyoghlig}_{2^e}})x+\zeta_{2^e}^{\text{ \textlyoghlig}_{2^e}+1}$ where $\text{ \textlyoghlig}_{2^e} \in \{1, \cdots, 2^e-1\}$ with $(\text{ \textlyoghlig}_{2^e} , 2^e)=1$. Also by Lemma \ref{order-zeta2e} we have either $\operatorname{o}_F(\zeta_{2^e})=2$ or $\operatorname{o}_F(\zeta_{2^e})=2^{e-1}$.	
	\begin{enumerate}
		\item is clear. 
		 \item Suppose that $\operatorname{o}_F(\zeta_{2^e})=2^{e-1}$. 
		By definition of the order of $\zeta_{2^e}$ we have that for all $j\in \{1, \cdots, e-2\}$, $\zeta_{2^e}^{2^j}\in F(\zeta_{2^e})\setminus F$. Therefore, $F(\zeta_{2^e})=F(\zeta_{2^e}^{2^j})$ for all $j\in \{1, \cdots, e-2\}$. Since $\zeta_{2^e}^{\text{ \textlyoghlig}_{2^e}+1} \in F$ then $\operatorname{o}_F(\zeta_{2^e})|\text{ \textlyoghlig}_{2^e}+1$ by Lemma \ref{minimalpolynomialzetan} which implies that $2^{e-1}|\text{ \textlyoghlig}_{2^e}+1$. So that $\text{ \textlyoghlig}_{2^e}+1\equiv 2^{e-1}s \mod 2^e$ where $s\in \{1, 2\}$. In particular, 
		\begin{enumerate}
			\item When $s=1$, then $\text{ \textlyoghlig}_{2^e} +1\equiv 2^{e-1}\mod 2^e$ so that $\text{ \textlyoghlig}_{2^e} \equiv 2^{e-1}-1\mod 2^e$. Thus, $\operatorname{min}(\zeta_{2^e}, F)=x^2-(\zeta_{2^e}+\zeta_{2^e}^{2^{e-1}-1})x+\zeta_{2^e}^{2^{e-1}}=x^2-(\zeta_{2^e}-\zeta_{2^e}^{-1})x-1$ and $\zeta_{2^e}-\zeta_{2^e}^{-1} \in F$.
			\item When $s=2$, then $\text{ \textlyoghlig}_{2^e} +1\equiv 2^e \mod 2^e$ which implies that $\text{ \textlyoghlig}_{2^e} \equiv -1 \mod 2^e$. Therefore, $\operatorname{min}(\zeta_{2^e}, F)=x^2-(\zeta_{2^e}+\zeta_{2^e}^{-1})x+1$ and $\zeta_{2^e}+\zeta_{2^e}^{-1} \in F$.
		\end{enumerate} 
	\end{enumerate}  
\end{proof}
We can deduce the following Corollary which determines when two quadratic cyclotomic extensions generated by $2$ power primitive roots of unity can be equal. 
\begin{corollary}\label{radical-e2}
	Let $e \in \mathbb{N}$. If $F(\zeta_{2^e})/F$ is a quadratic extension and $\operatorname{o}_F(\zeta_{2^e})=2$, 
	then we cannot have $F(\zeta_{2^e})=F(\zeta_{2^f})$ with $f>e$ except when $e=2$. 
\end{corollary}
\begin{proof}
	Suppose that $[F(\zeta_{2^e}):F]=2$ and $\operatorname{o}_F(\zeta_{2^e})=2$. 
	Using contradiction suppose that 
	$F(\zeta_{2^e})=F(\zeta_{2^f})$ with $f>e$ and $e>2$. Then $[F(\zeta_{2^f}):F]=2$ since $[F(\zeta_{2^e}):F]=2$. But $\operatorname{o}_F(\zeta_{2^f})\neq 2$ since otherwise, it would imply that $\zeta_{2^e}\in F$ which is a contradiction. By Theorem \ref{order-zeta2e}, $\operatorname{o}_F(\zeta_{2^f})= 2^{f-1}$ providing us with the result.
\end{proof}
From the following corollary, we learn that as soon as a quadratic cyclotomic extension is generated by two distinct $2$ power primitive roots of unity, we have that the cyclotomic extension is generated by $\zeta_4$. 
\begin{corollary}\label{lemma-equal}
	We suppose there is $e\in \mathbb{N}$ such that $\zeta_{2^e}\notin F$ and either $\zeta_{2^e}+\zeta_{2^e}^{-1}\in F$ or $\zeta_{2^e}-\zeta_{2^e}^{-1}\in F$.  Then $F(\zeta_{2^e})=F(\zeta_4)$. 
\end{corollary}

\begin{proof}
	Suppose that there is $e\in \mathbb{N}$ such that $\zeta_{2^e}\notin F$ and either $\zeta_{2^e}+\zeta_{2^e}^{-1}\in F$ or $\zeta_{2^e}-\zeta_{2^e}^{-1}\in F$. Then this implies that $\zeta_{2^e}$ is a root of an irreducible polynomial $x^2-(\zeta_{2^e}+\zeta_{2^e}^{-1})x+1$ or $x^2-(\zeta_{2^e}-\zeta_{2^e}^{-1})x-1$ over $F$. So that $[F(\zeta_{2^e}):F]=2$ and $\operatorname{o}_F(\zeta_{2^e})=2^{e-1}$ by Lemma \ref{valuesofkfor2^e}. That implies that $\zeta_4\notin F$. Since $[F(\zeta_4):F]=2$, thus we obtain $F(\zeta_{2^e})=F(\zeta_4)$. 
\end{proof}

We reach the main theorem of this section, which, on one hand, explicitly computes the minimal polynomials for the primitive root of unity that generates a quadratic cyclotomic extension. On the other hand, it compiles most of the results from this section. It is worthy noting that the coefficient of the minimal polynomial of a root of unity $\zeta_n$, defining a quadratic extension, can be entirely expressed in terms of the constants $\operatorname{o}_F(\zeta_n)$ and $\operatorname{d}_F(n)$. This provides a compelling proof of their importance in this narrative.

\begin{theorem}\label{theorem-valueofk}
Let $n\in \mathbb{N}$, $[F(\zeta_n): F]=2$, and $\sigma$ be the non-trivial element in $Gal( F(\zeta_n) /F)$. Then 
\begin{enumerate}
  \item when $n$ is odd, or $n$ is even and $\zeta_{2^{\text{\textepsilon}_n (2)}}\in F$,  we have
  $$F(\zeta_n ) = F( \zeta_{\operatorname{o}_F(\zeta_n)}), \ \sigma ( \zeta_n ) = \zeta_{\operatorname{d}_F(n)} \zeta_{\operatorname{o}_F(\zeta_n)}^{-1},$$ and
  $$\operatorname{min}( \zeta_n , F) = x^2 - \zeta_{\operatorname{d}_F(n)} \left ( \zeta_{ \operatorname{o}_F(\zeta_n)} +\zeta_{ \operatorname{o}_F(\zeta_n)}^{-1}\right)x + \zeta_{\operatorname{d}_F(n)}^2.$$
    \item when  $2\parallel \operatorname{o}_F(\zeta_n)$,
  $$F(\zeta_n ) = F( \zeta_{2^{\text{\textepsilon}_n (2)-1}\operatorname{o}_F(\zeta_n)}),\ \sigma ( \zeta_n ) = \zeta_{2\operatorname{d}_F(n)} \zeta_{\operatorname{o}_F(\zeta_n)}^{-1},$$ and
  $$\operatorname{min}( \zeta_n , F) =  x^2 - \zeta_{2\operatorname{d}_F(n)} \left(\zeta_{\operatorname{o}_F(\zeta_n)} - \zeta_{\operatorname{o}_F(\zeta_n)}^{-1}\right) x - \zeta_{2\operatorname{d}_F(n)}^2.$$
  \item when $\text{\textepsilon}_n (2)>2$, $2^{\text{\textepsilon}_n (2)-1}\parallel \operatorname{o}_F(\zeta_n)$, and $\zeta_{2^{\text{\textepsilon}_n (2)}} + \zeta_{2^{\text{\textepsilon}_n (2)}}^{-1} \in F$, we have $$F(\zeta_n ) = F( \zeta_{2\operatorname{o}_F(\zeta_n)})= F(\zeta_4), \ \sigma ( \zeta_n ) = -\zeta_{\operatorname{d}_F(n)} \zeta_{2\operatorname{o}_F(\zeta_n)}^{-1},$$
  and
  $$\operatorname{min}(\zeta_n, F) = x^2 - \zeta_{\operatorname{d}_F(n)} \left( \zeta_{ 2\operatorname{o}_F(\zeta_n)} +\zeta_{ 2\operatorname{o}_F(\zeta_n)}^{-1}\right) x + \zeta_{\operatorname{d}_F(n)}^2.$$
  \item  when $\text{\textepsilon}_n (2)>2$, $2^{\text{\textepsilon}_n (2)-1}\parallel \operatorname{o}_F(\zeta_n)$, and $\zeta_{2^{\text{\textepsilon}_n (2)}} - \zeta_{2^{\text{\textepsilon}_n (2)}}^{-1} \in F$, we have 
  $$F(\zeta_n ) = F( \zeta_{2\operatorname{o}_F(\zeta_n)})= F(\zeta_4), \ \sigma (\zeta_n ) = \zeta_{\operatorname{d}_F(n)} \zeta_{2\operatorname{o}_F(\zeta_n)}^{-1},$$ and
  $$\operatorname{min}( \zeta_n , F) = x^2 - \zeta_{\operatorname{d}_F(n)} \left( \zeta_{2 \operatorname{o}_F(\zeta_n)} - \zeta_{2 \operatorname{o}_F(\zeta_n)}^{-1}\right)x - \zeta_{\operatorname{d}_F(n)}^2.$$
\end{enumerate}
\end{theorem}

\begin{proof}
	\begin{enumerate}
		\item Suppose that either $n$ is odd, or $n$ is even and $\zeta_{2^{\text{\textepsilon}_n (2)}}\in F$.
		By Lemma \ref{minimalpolynomialzetan}, we know that $\operatorname{o}_F(\zeta_n)|{\text{\textlyoghlig}_n}+1$.  That implies that ${\text{\textlyoghlig}_n}\equiv -1 \mod \operatorname{o}_F(\zeta_n).$ Now,  since $\zeta_{\operatorname{d}_F(n)}\in F$, $\zeta_{\operatorname{d}_F(n)}=\sigma(\zeta_{\operatorname{d}_F(n)})=\sigma(\zeta_{n}^{\operatorname{o}_F(\zeta_n)}) =\zeta_{n}^{{\text{\textlyoghlig}_n}\operatorname{o}_F(\zeta_n)}=\zeta_{\operatorname{d}_F(n)}^{\text{\textlyoghlig}_n}$, and ${\text{\textlyoghlig}_n}\equiv 1 \mod \operatorname{d}_F(n).$ 
		By Corollary \ref{coprime}, we have $( \operatorname{o}_F(\zeta_n), \operatorname{d}_F(n) )=1$, and thus, we obtain $\zeta_n=\zeta_{\operatorname{d}_F(n)}\zeta_{\operatorname{o}_F(\zeta_n)}$. Thus, $\sigma ( \zeta_n ) =\sigma(\zeta_{\operatorname{d}_F(n)} \zeta_{\operatorname{o}_F(\zeta_n)}) = \zeta_{\operatorname{d}_F(n)} \zeta_{\operatorname{o}_F(\zeta_n)}^{-1}$ 
		The rest of the proof follows easily.
				\item Suppose that $2\parallel \operatorname{o}_F(\zeta_n)$.
		Then $\operatorname{o}_F(\zeta_{2^{\text{\textepsilon}_n (2)}})=2$, $\operatorname{o}_F(\zeta_n)= 2 \operatorname{o}_F (\zeta_{\operatorname{q}_n(2)})$ and $\operatorname{d}_F(n)=2^{{\text{\textepsilon}_n (2)}-1} \operatorname{d}_F(\operatorname{q}_n(2))$. Thus, we have $F(\zeta_{2^{\text{\textepsilon}_n (2)}})=F(\zeta_n)$. Therefore, since $\operatorname{o}_F(\zeta_{2^{\text{\textepsilon}_n (2)}})=2$ and $[F(\zeta_{2^{\text{\textepsilon}_n (2)}}):F]=2$, by Lemma \ref{valuesofkfor2^e}, we have $\sigma(\zeta_{2^{\text{\textepsilon}_n (2)}})=-\zeta_{2^{{\text{\textepsilon}_n (2)}}}$. By $(1)$ above we have $\sigma(\zeta_{\operatorname{o}_F(\zeta_{\operatorname{q}_n(2)})}) = \zeta_{\operatorname{o}_F(\zeta_{\operatorname{q}_n(2)})}^{-1}$ when $\operatorname{o}_F(\zeta_{\operatorname{q}_n(2)})\neq 1$ and trivially when $\operatorname{o}_F(\zeta_{\operatorname{q}_n(2)})= 1$. Since $\zeta_n= \zeta_{2^{\text{\textepsilon}_n (2)}}\zeta_{\operatorname{d}_F(\operatorname{q}_n(2))}\zeta_{\operatorname{o}_F({\zeta_{\operatorname{q}_n(2)}})}$, it follows that 
		$$\begin{array}{clll}
			\sigma(\zeta_n) &=\sigma(\zeta_{2^{\text{\textepsilon}_n (2)}})\sigma(\zeta_{\operatorname{d}_F(\operatorname{q}_n(2))})\sigma(\zeta_{\operatorname{o}_F({\zeta_{\operatorname{q}_n(2)}})}) \\
			&=-\zeta_{2^{\text{\textepsilon}_n (2)}}\zeta_{\operatorname{d}_F(\operatorname{q}_n(2))}\zeta_{ \operatorname{o}_F(\zeta_{\operatorname{q}_n(2)})}^{-1} =\zeta_{2\operatorname{d}_F(n)}\zeta_{\operatorname{o}_F(\zeta_n)}^{-1}
		\end{array}$$
		Since $\zeta_{\operatorname{d}_F(n)} \in F$, we have $F(\zeta_n) =F(\zeta_{2^{{\text{\textepsilon}_n(2)}-1}\operatorname{o}_F (\zeta_n)})$. The rest of the proof is straightforward.
		\item Suppose that $\text{\textepsilon}_n (2)>2, \ 2^{\text{\textepsilon}_n (2)-1}\parallel \operatorname{o}_F(\zeta_n) \text{ and }\zeta_{2^{\text{\textepsilon}_n (2)}} + \zeta_{2^{\text{\textepsilon}_n (2)}}^{-1} \in F. $ Then, $F(\zeta_{2^{\text{\textepsilon}_n (2)}})=F(\zeta_n)$ and $\operatorname{o}_F(\zeta_n)= 2^{\text{\textepsilon}_n (2)-1}\operatorname{o}_F( \zeta_{\operatorname{q}_n(2)}).$ Therefore, $\operatorname{d}_F(n)= 2 \operatorname{d}_F(\operatorname{q}_n(2))$.  
		Since $\zeta_{2^{\text{\textepsilon}_n (2)}}+\zeta_{2^{\text{\textepsilon}_n (2)}}^{-1}\in F$ by the assumption, therefore by Lemma \ref{valuesofkfor2^e} and Lemma \ref{minimalpolynomialzetan}, $\sigma(\zeta_{2^{\text{\textepsilon}_n (2)}})=\zeta_{2^{\text{\textepsilon}_n (2)}}^{-1}$.  
		We have $\sigma(\zeta_{\operatorname{o}_F(\zeta_{\operatorname{q}_n(2)})}) = \zeta_{\operatorname{o}_F(\zeta_{\operatorname{q}_n(2)})}^{-1}$.\\ 
		Since $(2^{\text{\textepsilon}_n (2)}, \operatorname{d}_F(\operatorname{q}_n(2)))=(\operatorname{d}_F(\operatorname{q}_n(2)), \operatorname{o}_F({\operatorname{q}_n(2)}))=(2^{\text{\textepsilon}_n (2)}, \operatorname{o}_F({\operatorname{q}_n(2)}))=1,$ we have that $\zeta_n=\zeta_{2^{\text{\textepsilon}_n (2)}}\zeta_{\operatorname{d}_F(\operatorname{q}_n(2))}\zeta_{\operatorname{o}_F({\zeta_{\operatorname{q}_n(2)}})}.$  So that,
		$$\begin{array}{clll}
			\sigma(\zeta_n) 
			&=&\sigma(\zeta_{2^{\text{\textepsilon}_n (2)}}\zeta_{\operatorname{d}_F(\operatorname{q}_n(2))}\zeta_{\operatorname{o}_F({\zeta_{\operatorname{q}_n(2)}})})\\
			&=&\zeta_{2^{\text{\textepsilon}_n (2)}}^{-1}\zeta_{\operatorname{d}_F(\operatorname{q}_n(2))}\zeta_{\operatorname{o}_F(\zeta_{\operatorname{q}_n(2)})}^{-1}		=-\zeta_{\operatorname{d}_F(n)}\zeta_{2\operatorname{o}_F(\zeta_n)}^{-1} 
		\end{array} $$
Since $\zeta_{\operatorname{d}_F(\operatorname{q}_n(2))} \in F$, we have $F(\zeta_n) =F(\zeta_{2\operatorname{o}_F (\zeta_n)})$. Since ${\text{\textepsilon}_n (2)}>2$, we have $4|n$, and $\zeta_n^{2^{{\text{\textepsilon}_n (2)}-2} m}=\zeta_4 \in F(\zeta_n) \setminus F$, since $2^{{\text{\textepsilon}_n (2)}-1} \parallel \operatorname{o}_F (\zeta_n)$ by assumption. Therefore, $F(\zeta_n) = F(\zeta_4)$. The rest of the proof follows easily.

		\item  Suppose that $2^{{\text{\textepsilon}_n (2)}}\parallel n, {\text{\textepsilon}_n (2)}>2,  \ 2^{{\text{\textepsilon}_n (2)}-1}\parallel \ \operatorname{o}_F(\zeta_n) \text{ and }
		\zeta_{2^{\text{\textepsilon}_n (2)}} - \zeta_{2^{\text{\textepsilon}_n (2)}}^{-1} \in F.$ By $(2)$, $\operatorname{o}_F(\zeta_{n})=2^{{\text{\textepsilon}_n (2)}-1}\operatorname{o}_F( \zeta_{\operatorname{q}_n(2)})$ and $\operatorname{d}_F(n) =2 \operatorname{d}_F(\operatorname{q}_n(2))$. Since $\zeta_{2^{\text{\textepsilon}_n (2)}} - \zeta_{2^{\text{\textepsilon}_n (2)}}^{-1} \in F$ by assumption, therefore we have $\sigma(\zeta_{2^{\text{\textepsilon}_n (2)}})=-\zeta_{2^{{\text{\textepsilon}_n (2)}}}^{-1}$ by Lemma \ref{valuesofkfor2^e}. We have $\sigma(\zeta_{\operatorname{o}_F(\zeta_{\operatorname{q}_n(2)})}) = \zeta_{\operatorname{o}_F(\zeta_{\operatorname{q}_n(2)})}^{-1}$.  It follows that 
		$$\begin{array}{lllll}
			\sigma(\zeta_n) &=\sigma(\zeta_{2^{\text{\textepsilon}_n (2)}}\zeta_{\operatorname{d}_F(\operatorname{q}_n(2))}\zeta_{\operatorname{o}_F(\zeta_{\operatorname{q}_n(2)})})\\
			&=-\zeta_{2^{\text{\textepsilon}_n (2)}}^{-1}\zeta_{\operatorname{d}_F(\operatorname{q}_n(2))}\zeta_{\operatorname{o}_F(\zeta_{\operatorname{q}_n(2)})}^{-1} =\zeta_{\operatorname{d}_F(n)}\zeta_{2\operatorname{o}_F(\zeta_n)}^{-1}.	
		\end{array}$$  
		We prove that $F(\zeta_n ) = F( \zeta_{2\operatorname{o}_F(\zeta_n)})= F(\zeta_4)$, similarly as in $(3)$ and the rest of the proof is easily derived.
			\end{enumerate}
\end{proof}

\begin{remark}
Keeping the notation of the previous theorem \ref{theorem-valueofk}:
\begin{enumerate} 
\item By Lemma \ref{minimalpolynomialzetan}, we also know that the minimal polynomial of $\zeta_n$ is of the form
$$x^2 - (\zeta_n + \zeta_n^{\text{\textlyoghlig}_n})x + \zeta_n^{\text{\textlyoghlig}_n+1}.$$
	From the computations of $\sigma( \zeta_n)$ in Theorem \ref{theorem-valueofk}, we can easily explicitly compute $\text{\textlyoghlig}_n$ modulo $n$ using the Chinese remainder theorem.
	\item From Theorem \ref{theorem-valueofk}, we can deduce the following \begin{itemize}
  \item $\operatorname{d}_F(n)=1$ or, $\operatorname{d}_F(n)=2$ and $\zeta_{2^{\text{\textepsilon}_n (2)}}+\zeta_{2^{\text{\textepsilon}_n (2)}}^{-1}\in F$ if and only if $\text{\textlyoghlig}_n \equiv -1\mod n$.
  \item  $\zeta_n$ is a radical generator for $F(\zeta_n)/F$ if and only if $2= \operatorname{o}_F(\zeta_n)$. In which case, $\text{\textlyoghlig}_n \equiv 1+\frac{n}{2} \mod n$ and $\operatorname{min}( \zeta_n , F) = x^2 - \zeta_{\operatorname{d}_F(n)}$.
\end{itemize}
 \item In a characteristic different from $2$, the irreducibility of the polynomial $x^2-(a\pm a^{-1})x \pm1$ over the field $F$ is equivalent to the condition $a\mp a^{-1} \notin F$. This can be proven easily by computing the discriminant of this quadratic polynomial. In particular, when we know that a $p$ power primitive root of unity has a degree less than or equal to $2$, this permits us to check whether it is of degree $2$ or not.
	\end{enumerate}
\end{remark}

\section{About the structure of the sets of quadratic cyclotomic extensions} 
\subsection{Maximal cyclotomic extensions}
We start this section by redefining one of the key constants mentioned in the introduction.
\begin{definition}\label{max-l}
 We define $\ell_{p^\infty_F}$ as the integer such that
 $$\ell_{p^\infty_F}=\begin{cases}
				\operatorname{max}\{k \in \mathbb{N} \cup \{ 0 \} |\zeta_{p^k}\in F\} & \text{when it exists}; \\
				\infty & \text{otherwise}.
			\end{cases} $$
\end{definition}
We also define the set of the roots of unity.
\begin{definition}\label{mu_infinity}
	\begin{enumerate}
		\item We define $\mu_{\infty}:=\bigcup\limits^{\infty}_{n=1}\mu_n$ to be the set of the roots of unity in $\overline{F}$. 
		\item We define $\mu_{p^\infty}:=\bigcup\limits^{\infty}_{k=1}\mu_{p^k}$ to be the set of the $p$ power roots of unity in $\overline{F}$. 
		\item We define $\mu_{\infty_F}$ (resp. $\mu_{p^{\infty}_F}$) 
		to be the set of roots of unity (resp.  $p$ power roots of unity where $k \in \mathbb{N}$ and $p$ is prime number) in $F$.
		\item We define $\mu_{2 \infty+1}:=\bigcup\limits^{\infty}_{n=1}\mu_{2n+1}$. 
	\end{enumerate}
\end{definition}
\begin{remark}\label{muinfinity-element}
 We make the following observations.
	\begin{enumerate}
		\item $\mu_{\infty}=\bigvee \limits_{p \in \mathbb{P}}\mu_{p^{\infty}}:=\{ \prod_{p\in S} \zeta_{p^{e_p}} | S \finsub \mathbb{P} \text{ and } e_p \in \mathbb{N},\ \forall p \in S \}$. 
			\item $\mu_{p^{\infty}_F}=\mu_{p^{\ell_{p^\infty_F}}}.$ 
		\item $\mu_{\infty_F}=\bigcup\limits_{n\in \mathbb{N}}\mu_{\operatorname{d}_F(n)}= \mu_{\infty}\cap F =\bigvee \limits_{p \in \mathbb{P}}\mu_{p^{\ell_{p_F^\infty}}} :=\{ \prod_{p\in S} \zeta_{p^{e_p}} | S \finsub \mathbb{P} \text{ and } e_p \leq {\ell_{p_F^\infty}},\ \forall p \in S \}$. 
		\item The field obtained by adjoining to $F$ all roots of unity (resp. all the $p$ power primitive roots of unity) in $\overline{F}$ denoted by $F(\mu_\infty)$ (resp. $F(\mu_{p^\infty})$) can be defined as the intersection of all the subfields containing $F$ and $\mu_\infty$ (resp. $F$ and $\mu_{p^\infty}$). 
	\end{enumerate}

\end{remark}

The following zero map will permit us to describe the set of degree $2$ cyclotomic elements in terms of equaliser. 
\begin{definition}\label{zero-map}
Let $n \in \mathbb{N}$. 
	We define 
	$$ \begin{array}{clllll} 0_{\mu_{\infty}}:& \mu_{\infty} & \rightarrow & \frac{F(\mu_{\infty})}{F}   \\ 
		& \zeta_{n} & \mapsto & [0]_F
	\end{array} $$ 
\end{definition}

\subsection{The function $\operatorname{t}_F$} 

 To initiate our discussion, we introduce a map that facilitates the understanding of the set of quadratic cyclotomic extensions. 
\begin{definition}\label{tnf}
 We define the map $\operatorname{t}_F$ from $\mathbb{N}$ to $\mathbb{N}$ such that given a prime number $p$ and $e\in \mathbb{N}$, $\operatorname{t}_F(p^e)$ is equal to
$$\left\{ \begin{array}{clll}p^e & \text{when }\  p \ \text{ is odd and } \operatorname{o}_F(\zeta_{p^e})\neq1; \\
2^e & \text{when } \ p=2 \ \text{and } \ \operatorname{o}_F(\zeta_{2^e})>2;\\
2 & \text{when} \ p=2 \text{ and } \ \operatorname{o}_F(\zeta_{2^e})=2;\\
1 & \text{when} \ \operatorname{o}_F(\zeta_{p^e})=1.
\end{array} \right.$$ 
Moreover, $\operatorname{t}_F(n)=\prod \limits_{p \in \mathbb{P},p|n} \operatorname{t}_F(p^{\text{\textepsilon}_n(p)})$, for any $n \in \mathbb{N}$. 

\end{definition}
The following remark will be particularly useful in the proofs of Lemma \ref{property-q}, Lemma \ref{prime-equaliser}, and Lemma \ref{vp-propertyQ}.

\begin{remark}\label{tnf-remark} 
When $[F(\zeta_n) : F]=2$, we obtain the following:	
\begin{enumerate}
\item For all $n \in \mathbb{N}$, $$\operatorname{t}_F(n) =\left\{ \begin{array}{clll} 2\operatorname{o}_F(\zeta_n)& \ \text{when} \ \operatorname{o}_F(\zeta_{2^{\text{\textepsilon}_n(2)}})>2; \\
		\operatorname{o}_F(\zeta_n) & \text{otherwise}. 
	\end{array} \right.$$  This follows from Lemma \ref{remark1} and Lemma \ref{valuesofkfor2^e}.
\item Consider the definition above. If $f<e$, we have $\operatorname{t}_F(p^f)|\operatorname{t}_F(p^e)$. In particular,
	\begin{itemize}
		\item $\operatorname{t}_F(p^e)=p^{e-f}\operatorname{t}_F(p^f)$ when $p^e|n$, $p$ is odd and $\operatorname{o}_F(\zeta_{p^f})\neq1$, or $p=2$ and $\operatorname{o}_F(\zeta_{2^f})>2$,
		\item $\operatorname{t}_F(2^e)=2^{e-f}\operatorname{t}_F(2^f)$ when $2^e|n$, $\operatorname{o}_F(\zeta_{2^e})>2$ and $\operatorname{o}_F(\zeta_{2^f})>2$, 
		\item $\operatorname{t}_F(2^e)=2^{e-1}\operatorname{t}_F(2^f)$ when $2^e|n$, $\operatorname{o}_F(\zeta_{2^e})>2$ and $\operatorname{o}_F(\zeta_{2^f})=2$, 
		\item $\operatorname{t}_F(2^e)=2^e\operatorname{t}_F(2^f)$ when $2^e|n$, $\operatorname{o}_F(\zeta_{2^e})>2$ and $\operatorname{o}_F(\zeta_{2^f})=1$,
		\item $\operatorname{t}_F(2^e)=2\operatorname{t}_F(2^f)$ when $2^e|n$, $\operatorname{o}_F(\zeta_{2^e})=2$,
		\item $\operatorname{t}_F(p^e)=\operatorname{t}_F(p^f)=1$ when $p^e|n$, $\operatorname{o}_F(\zeta_{p^e})=1$. 
	\end{itemize}
\end{enumerate}
\end{remark}

\subsection{Order 2 primitive roots of unity} 
We start the section with the following definition: 
\begin{definition}
\begin{enumerate}
\item We define $\mathscr{G}_{2,\overline{F}}$ (resp. $\mathscr{G}_{2, \overline{F}}^p$) to be the set of primitive roots of unity (resp. $p$ power primitive roots of unity) of order $2$ in $F$. 
\item We define $\mathscr{R}_{2,\overline{F}}$ (resp. $\mathscr{R}_{2, \overline{F}}^p$) the set of cyclotomic extension $F(\zeta_n)$ such that $\zeta_n \in \mathscr{G}_{2,\overline{F}}$ (resp. $\zeta_n \in \mathscr{R}_{2,\overline{F}}$).
\end{enumerate}
 \end{definition}

The following lemma characterizes when a general primitive root of unity is of order $2$ over $F$. 
\begin{lemma}\label{zetanradical}
		Let $n\in \mathbb{N}$. The following assertions are equivalent:
		\begin{enumerate}
			\item $\zeta_n\in \mathscr{G}_{2,\overline{F}}$.
			\item $\operatorname{o}_F(\zeta_{n})=\operatorname{o}_F(\zeta_{2^{\text{\textepsilon}_n(2)}})=2$.
			\item $\zeta_n\notin F$, $\zeta_{\operatorname{t}_F(2^{\text{\textepsilon}_n(2)})}-\zeta_{\operatorname{t}_F(2^{\text{\textepsilon}_n(2)})}^{-1}=0$ and $\text{\texthtq}_n(2) |\operatorname{d}_F(n)$.
		\end{enumerate} 
	\end{lemma}
	\begin{proof}
		$(1)\Longrightarrow (2)$ Suppose that $\zeta_n\in \mathscr{G}_{2,\overline{F}}$. That is $\operatorname{o}_F(\zeta_n)=2$. Moreover, by \cite[Theorem 4.5]{conrad2014orders} we have $2=
		\operatorname{o}_F(\zeta_n)=\operatorname{o}_F(\zeta_{2^{\text{\textepsilon}_n(2)}}\zeta_{\operatorname{q}_n(2)})=\operatorname{o}_F(\zeta_{2^{\text{\textepsilon}_n(2)}})\operatorname{o}_F(\zeta_{\operatorname{q}_n(2)})$. As a consequence, $\operatorname{o}_F(\zeta_{\operatorname{q}_n(2)})=1$ and $\operatorname{o}_F(\zeta_{2^{\text{\textepsilon}_n(2)}})=2$. 
		
		$(2) \Longrightarrow (3)$ Suppose that $\operatorname{o}_F(\zeta_{n})=\operatorname{o}_F(\zeta_{2^{\text{\textepsilon}_n(2)}})=2$. 
		
		Since $2=\operatorname{o}_F(\zeta_n)=\operatorname{o}_F(\zeta_{2^{\text{\textepsilon}_n(2)}})\operatorname{o}_F(\zeta_{\operatorname{q}_n(2)})$, we have $\zeta_{\operatorname{q}_n(2)}\in F$ so that ${\operatorname{q}_n(2)}|\operatorname{d}_F(n)$. Moreover, $\operatorname{o}_F(\zeta_{2^{\text{\textepsilon}_n(2)}})=2$ implies $\operatorname{t}_F(2^{\text{\textepsilon}_n(2)})=2$ by Definition \ref{tnf}. Hence, $\zeta_{\operatorname{t}_F(2^{\text{\textepsilon}_n(2)})}-\zeta_{\operatorname{t}_F(2^{\text{\textepsilon}_n(2)})}^{-1}=\zeta_2-\zeta_2^{-1}=0$.
		
		$(3)\Longrightarrow (1)$. Suppose that $\zeta_n\notin F$, $\zeta_{\operatorname{t}_F(2^{\text{\textepsilon}_n(2)})}-\zeta_{\operatorname{t}_F(2^{\text{\textepsilon}_n(2)})}^{-1}=0$ and $\text{\texthtq}_n(2) |\operatorname{d}_F(n)$. Since $\text{\texthtq}_n(2) |\operatorname{d}_F(n)$, then $\zeta_{\operatorname{q}_n(2)}\in F$. Also, $\zeta_{\operatorname{t}_F(2^{\text{\textepsilon}_n(2)})}-\zeta_{\operatorname{t}_F(2^{\text{\textepsilon}_n(2)})}^{-1}=0$ implies that $\zeta_{\operatorname{t}_F(2^{\text{\textepsilon}_n(2)})}^2=1$. Hence, we have either $\operatorname{t}_F(2^{\text{\textepsilon}_n(2)})=1$ or $\operatorname{t}_F(2^{\text{\textepsilon}_n(2)})=2$. But $\operatorname{t}_F(2^{\text{\textepsilon}_n(2)})=1$ is impossible since it implies $\zeta_{2^{\text{\textepsilon}_n(2)}}\in F$ by Definition \ref{tnf} and thus, $\zeta_n\in F$. Thus, $\operatorname{t}_F(2^{\text{\textepsilon}_n(2)})=2$ and $\operatorname{o}_F(\zeta_{2^{\text{\textepsilon}_n(2)}})=2$. Then $\operatorname{o}_F(\zeta_n)=2$. Therefore,  $\zeta_n\in \mathscr{G}_{2,\overline{F}}$ concluding the proof. 
\end{proof}

\begin{remark} 
From the previous lemma, we can further prove that given a prime number $p$ such that $\zeta_{p^e}\notin F$, $\zeta_{p^e}\in \mathscr{G}_{2,\overline{F}}$ if and only if $p=2$ and $\zeta_{\operatorname{t}_F(p^e)}-\zeta_{\operatorname{t}_F(p^e)}^{-1}=0$. 
\end{remark}

We next prove that the collection of cyclotomic elements that define radical extensions of degree $2$ forms a group provided that the set is non-empty. 
\begin{theorem}\label{g2f} The set of $p$-primitive roots of unity of order $2$ in $F$ is given by
	$$\mathscr{G}_{2,\overline{F}}=\left\{ \begin{array}{clll}\emptyset & \text{ when} \ \ell_{2^\infty_F}=\infty; \\
		\mathcal{P}_{2^{\ell_{2^\infty_F}+1}}\bigvee{\mu_{{2\infty+1}_F}} & \text{ otherwise.} \end{array} \right.$$ In particular, $F\left(\zeta_{2^{\ell_{2^\infty_F}+1}}\right)=F\left(\zeta_{2^{\ell_{2^\infty_F}+1}m}\right)$ for all odd integer $m\in \mathbb{N}$ such that $\zeta_m\in F$. When $\ell_{2^\infty_F}< \infty$, then $\mathscr{G}_{2,\overline{F}}$ is a group when endowed with the binary operation $\smallstar$ given by $({\zeta^{k_1}_{2^{\ell_{2^\infty_F}+1}}} \zeta_{m_1} ) \smallstar ({\zeta^{k_2}_{2^{\ell_{2^\infty_F}+1}}}\zeta_{m_2})={\zeta^{k_1k_2}_{2^{\ell_{2^\infty_F}+1}}}\zeta_{m_1} \zeta_{m_2}$. In particular, the set of primitive  roots of unity of order $2$ in $F$ is given by
	$$\mathscr{G}_{2,\overline{F}}^p=\left\{ \begin{array}{clll}\emptyset & \text{ when} \ p \ \text{is odd, or } p=2 \text{ and }  \ell_{2^\infty_F}=\infty\text{}; \\
		\mathcal{P}_{2^{\ell_{2^\infty_F}+1}}& \text{ otherwise.} \end{array} \right.$$
\end{theorem}
\begin{proof}
	By contradiction, we can easily prove that $\mathscr{G}_{2,\overline{F}}=\emptyset$ when $\ell_{2^\infty_F}=\infty$. When $\ell_{2^\infty_F}<\infty$, the result follows from Corollary \ref{zetanradical}. 
\end{proof}
From the above Theorem, we can obtain easily the following Corollary describing the set of cyclotomic extensions generated by a primitive root of unity of order $2$.
\begin{corollary}\label{allradical}
The sets of quadratic cyclotomic extensions $\mathscr{R}_{2,\overline{F}}$ that admit a root of unity in $\overline{F}$ as a radical generator over $F$ can be described as follows:
\[
\mathscr{R}_{2,\overline{F}}=
\begin{cases}
\emptyset & \text{when}\ \ell_{2^\infty_F}=\infty; \\
\left\{F\left(\zeta_{2^{\ell_{2^\infty_F}+1}}\right)\right\} & \text{otherwise}.
\end{cases}
\]
\end{corollary}

\begin{remark} 
$$\mathscr{R}_{2,\overline{F}} = \mathscr{R}_{2,\overline{F}}^2$$
\end{remark} 
\subsection{The property $\mathcal{C}_2$}
We will see later that the structure of the sets of degree $2$ cyclotomic elements will be affected by the property $\mathcal{C}_2$ defined in the next Definition-Lemma. These sets will take $3$ forms depending on the fixed base field and one of these forms relies on this property.  
\begin{deflem} \label{property-q}
	We say an integer $e\in \mathbb{N}$ has {\sf property $\mathcal{C}_2$} if $\zeta_{2^e}\notin F$, $\operatorname{t}_F(2^e)\neq 2$ and $\zeta_{2^e}-\zeta_{2^e}^{-1}\in F$.  
We say that $F$ has {\sf property $\mathcal{C}_2$} if there is $e\in \mathbb{N}$ which has property $\mathcal{C}_2$. 
	When an integer $e$ with property $\mathcal{C}_2$ exists, it is unique and we denote it $c_2$. Moreover, when there is a natural number $e$ which has property $\mathcal{C}_2$ then, 
	\begin{enumerate}
		\item $\operatorname{o}_F(\zeta_{2^e})=2^{e-1}$ and $\operatorname{min}(\zeta_{2^e}, F)=x^2-(\zeta_{2^e}-\zeta_{2^e}^{-1})x-1$;
		\item for all $f<e$, we have $\zeta_{\operatorname{t}_F(2^f)}+ \zeta_{\operatorname{t}_F(2^f)}^{-1}\in F$. 
	\end{enumerate}
\end{deflem}
\begin{proof}
	Suppose that there exists $e$ which has property $\mathcal{C}_2$ 
	\begin{enumerate}
		\item  This follows easily from Lemma \ref{valuesofkfor2^e}. 
		\item We want to prove that for all $f<e$,  $\zeta_{\operatorname{t}_F(2^f)}+ \zeta_{\operatorname{t}_F(2^f)}^{-1}\in F$. By definition, $\operatorname{o}_F(\zeta_{2^e})>2$. Hence, by Definition \ref{tnf}, $\operatorname{t}_F(2^e)=2^e$ and $\zeta_{\operatorname{t}_F(2^e)}-\zeta_{\operatorname{t}_F(2^e)}^{-1}\in F$. It follows that  $\zeta_{\operatorname{t}_F(2^e)}^2+\zeta_{\operatorname{t}_F(2^e)}^{-2}\in F$, since $ \zeta_{\operatorname{t}_F(2^e)}^2+\zeta_{\operatorname{t}_F(2^e)}^{-2}= (\zeta_{\operatorname{t}_F(2^e)}-\zeta_{\operatorname{t}_F(2^e)}^{-1})^2 +2$ and $(\zeta_{\operatorname{t}_F(2^e)}-\zeta_{\operatorname{t}_F(2^e)}^{-1})^2\in F$. Now let $f<e$. Then applying Remark \ref{tnf-remark}, we obtain that $\operatorname{t}_F(2^f) |\operatorname{t}_F(2^e)$. Therefore, by Lemma \ref{waring} we have obtained in all cases  that $\zeta_{\operatorname{t}_F(2^f)}+\zeta_{\operatorname{t}_F(2^f)}^{-1}\in F$. Now, $f$ cannot have property $\mathcal{C}_2$, since otherwise $\zeta_{\operatorname{t}_F(2^f)}+\zeta_{\operatorname{t}_F(2^f)}^{-1}+\zeta_{\operatorname{t}_F(2^f)}-\zeta_{\operatorname{t}_F(2^f)}^{-1}=2\zeta_{\operatorname{t}_F(2^f)}\in F$ and this would contradict $\operatorname{o}_F(\zeta_{2^e})=2^{e-1}$ when $\operatorname{t}_F(2^f)\neq 2$. 
	\end{enumerate}
	We lastly prove that there is only one element with property $\mathcal{C}_2$ when it exists. Let $e$ be such an element. We have already proven that any natural number smaller than $e$ cannot have property $\mathcal{C}_2$. We use contradiction to prove that the same holds for any natural number greater than $f$. We assume that there exists a natural number $k$ such that $k>e$ with property $\mathcal{C}_2$. Then, by Lemma \ref{waring} and Remark \ref{tnf-remark}, we have $\zeta_{\operatorname{t}_F(2^e)}+\zeta_{\operatorname{t}_F(2^e)}^{-1}\in F$ since $e<k$. That implies that $\zeta_{\operatorname{t}_F(2^e)}\in F$ since $\zeta_{\operatorname{t}_F(2^e)}+\zeta_{\operatorname{t}_F(2^e)}^{-1}+\zeta_{\operatorname{t}_F(2^e)}-\zeta_{\operatorname{t}_F(2^e)}^{-1}=2\zeta_{\operatorname{t}_F(2^e)}\in F$. So that $\zeta_{2^e}\in F$, since, by assumption $\operatorname{t}_F(2^e)\neq 2$, then $\operatorname{o}_F(\zeta_{2^e})\neq 2$ and thus $\operatorname{t}_F(2^e)=2^e$ by Definition \ref{tnf} and Lemma \ref{order-zeta2e}, leading to a contradiction. Therefore, $e$ is unique. Hence, the proof is completed. 
\end{proof}
\subsection{Describing the set of quadratic cyclotomic as an equilizer} 
The forthcoming sections of the paper will primarily focus on characterizing sets of quadratic cyclotomic extensions, along with the set of roots of unity that define them. We will adopt the following notation for these sets.
\begin{definition}\label{quadratic-cyclotomic} 
		\begin{enumerate}
			\item We denote $\mathscr{C}_{2,\overline{F}}$ (resp. $\mathscr{C}_{2,\overline{F}}^p$) as the sets of quadratic cyclotomic extensions ($p$ power cyclotomic extensions) over $F$ in $\overline{F}$.
			\item We denote $\mathscr{M}_{2,\overline{F}}$ (resp. $\mathscr{M}_{2,\overline{F}}^p$) as the sets of roots of unity $\zeta$ in $\overline{F}$ such that $F(\zeta)\in \mathscr{C}_{2, \overline{F}}$ (resp. $F(\zeta)\in \mathscr{C}_{2, \overline{F}}^p$).
		\end{enumerate}
	
\end{definition}

\begin{remark}\label{iso-cyclotomic}
	We make the following observations about the sets defined above:
	
	\begin{enumerate}
		\item There exist natural mappings from $\mathscr{M}_{2,\overline{F}}$ (resp. $\mathscr{M}^p_{2,\overline{F}}$)  to $\mathscr{C}_{2,\overline{F}}$ (resp. $\mathscr{C}^p_{2,\overline{F}}$), which simply send a root of unity $\zeta$ to $F(\zeta)$. We note that these mappings are surjective but not one-to-one.
		
		\item Throughout this paper, we work within a fixed algebraic closure of $F$. As a consequence, given any $\zeta_n$ and $\zeta_m$ in $\overline{F}$ where $n, m\in \mathbb{N}$, the equality $F(\zeta_n) = F(\zeta_m)$ is equivalent to asserting that $F(\zeta_n)$ and $F(\zeta_m)$ are $F$-isomorphic. As a consequence, $\mathscr{C}_{2,\overline{F}}$ also corresponds to the set of cyclotomic extensions up to isomorphism.
	\end{enumerate}
\end{remark}
In this section, we describe the sets of quadratic cyclotomic extensions by means of an equalizer of the zero map and another single map. This map together with the constant $\ell_{p^\infty_F}$ characterizes the set of primitive roots of unity defining quadratic extensions. It functions as a synthesis of all the findings of the previous sections into a single mapping. The map in question is defined in the next definition and is constructed using Theorem \ref{theorem-valueofk}.

\begin{definition}\label{equaliser-general}
We define the map $$ \begin{array}{clll} \kappa_F:& \mu_\infty & \rightarrow & \frac{F(\mu_\infty)}{{}_{F}}   \\ 
		& \zeta_n & \mapsto & \begin{cases}[\zeta_{\operatorname{t}_F(n)} + \zeta_{\operatorname{t}_F(n)}^{-1}]_F \ \ \text{if}\ \operatorname{o}_F(\zeta_{2^{\text{\textepsilon}_n(2)}})\neq 2 \ \text{ and } \ \text{\textepsilon}_n(2)\neq c_2;\\
			 [\zeta_{\operatorname{t}_F(n)} - \zeta_{\operatorname{t}_F(n)}^{-1}]_F\ \ \text{if}\ \operatorname{o}_F(\zeta_{2^{\text{\textepsilon}_n(2)}})\neq 2 \ \text{ and } \ \text{\textepsilon}_n(2)= c_2;\\
			 [\zeta_{2^{ \text{\textepsilon}_n(2)}}(\zeta_{\operatorname{t}_F(n)} - \zeta_{\operatorname{t}_F(n)}^{-1})]_F \ \ \text{ otherwise. } \end{cases}
	\end{array} $$
	where $\frac{F(\mu_\infty)}{F}$ is the quotient sets when both $F(\mu_\infty)$ and $F$ are seen as additive groups and $[a]_F$ is the coset of $a$ in that quotient.
We denote $\kappa_F^p$ to be $\kappa_F|_{\mu_{p^\infty}}$. 
\end{definition}
\begin{remark}
When $p$ is an odd prime, we have $ \kappa_F^p$ sends $\zeta_{p^e}$ to $[\zeta_{p^e}+ \zeta_{p^e}^{-1}]_F$. 
\end{remark} 

We are now ready to prove Theorem \ref{Equali}, as stated in the introduction, which we restate here for the convenience of the reader.
\begin{introthm1}{Theorem}  \label{Equali}
The set $\mathscr{M}_{2, \overline{F}}$ of primitive roots of unity $\zeta$ in $\overline{F}$ such that $[F(\zeta):F]=2$ can be expressed as the following equalizer:
\[
\mathscr{M}_{2, \overline{F}} = \operatorname{Eq}( \kappa_{F}, 0_{\mu_\infty}) \setminus \mu_{\infty_F}.
\]
\end{introthm1}

\begin{proof}
	We start by proving that $\mathscr{M}_{2, \overline{F}}\subseteq \operatorname{Eq}( \kappa_{F}, 0_{\mu_\infty})\setminus\mu_{\infty_F}$. Let $\zeta_n\in\mathscr{M}_{2,\overline{F}}$. Then $[F(\zeta_n):F]=2$ by Definition \ref{quadratic-cyclotomic}. That implies that $\zeta_n\notin F$ so that $\zeta_n\notin \mu_{\infty_F}$. We obtain that $\zeta_n \in \operatorname{Eq}( \kappa_{F}, 0_{\mu_\infty})$, by Theorem \ref{theorem-valueofk}, the definition of $\operatorname{t}_F$ and since $\zeta_{\operatorname{d}_{F} (n)} \in F$.
	
	Conversely, let $\zeta_n\in \operatorname{Eq}( \kappa_{F}, 0_{\mu_\infty})\setminus\mu_{\infty_F}$.We have $[F(\zeta_n):F]\geq 2$ since $\zeta_n\notin F$. We want to prove that $[F(\zeta_n):F]=2$.  
	Then by Definition \ref{tnf}, $\operatorname{t}_F(\operatorname{q}_n(2))|{\operatorname{q}_n(2)}$ and $\operatorname{t}_F(\operatorname{q}_n(2))r={\operatorname{q}_n(2)}$, for some $(r, \operatorname{t}_F(\operatorname{q}_n(2)))=1$. We have $\zeta_r\in F$. Also, we have $\zeta_n=\zeta_{2^{\text{\textepsilon}_n(2)}}\zeta_{\operatorname{t}_F(\operatorname{q}_n(2))}\zeta_{r}$. It then follows that $F(\zeta_n)=F(\zeta_{2^{\text{\textepsilon}_n(2)}\operatorname{t}_F(\operatorname{q}_n(2))})$ since $\zeta_r\in F$. We set $s:=2^{\text{\textepsilon}_n(2)}\operatorname{t}_F(\operatorname{q}_n(2))$. We note that $\zeta_s\notin F$ otherwise it contradicts the fact that $F(\zeta_s)=F(\zeta_n)$ and $\zeta_n\notin F$. It suffice to prove that $[F(\zeta_s):F]=2$ in order to prove that $[F(\zeta_n):F]=2$. 
	This follows from the following observations:
\begin{itemize} 
\item When $n$ is odd or $\operatorname{o}_F(\zeta_{2^{\text{\textepsilon}_n(2)}})>2$ and $\text{\textepsilon}_n(2)\neq c_2$, we have $s=\operatorname{t}_F(n)$  and $\zeta_{\operatorname{t}_F(n)}+\zeta_{\operatorname{t}_F(n)}^{-1}\in F$ by Definition \ref{equaliser-general}. 
\item When $\operatorname{o}_F(\zeta_{2^{\text{\textepsilon}_n(2)}})>2$ and $\text{\textepsilon}_n(2)=c_2$, we have $\operatorname{t}_F(n)=s$ and $\zeta_{\operatorname{t}_F(n)}-\zeta_{\operatorname{t}_F(n)}^{-1}\in F$ by Definition \ref{equaliser-general}. 
\item When n is even and $\zeta_{2^{\text{\textepsilon}_n(2)}}\in F$, we have $\operatorname{t}_F(n)=\operatorname{t}_F(\operatorname{q}_n(2))$, $F(\zeta_s)=F(\zeta_{\operatorname{t}_F(\operatorname{q}_n(2))})$ and $\zeta_{\operatorname{t}_F(n)}+\zeta_{\operatorname{t}_F(n)}^{-1}\in F$.
\item When $\operatorname{o}_F(\zeta_{2^{\text{\textepsilon}_n(2)}})=2$ we have $\operatorname{t}_F(n)=2\operatorname{t}_F(\operatorname{q}_n(2))$, $F(\zeta_s)=F(\zeta_{2^{\text{\textepsilon}_n(2)}}\zeta_{\operatorname{t}_F(\operatorname{q}_n(2))})=F(\zeta_{2^{\text{\textepsilon}_n(2)}}\zeta_{\operatorname{t}_F(n)})$,  
and $\zeta_{2^{\text{\textepsilon}_n(2)}}\zeta_{\operatorname{t}_F(n)}$ is a root of the polynomial $x^2-\zeta_{2^{\text{\textepsilon}_n(2)}}(\zeta_{\operatorname{t}_F(n)}-\zeta_{\operatorname{t}_F(n)}^{-1})x+\zeta_{2^{\text{\textepsilon}_n(2)}}^2$ over $F$, by Definition \ref{equaliser-general}. 
\end{itemize} 
Therefore, $\operatorname{Eq}( \kappa_{F}, 0_{\mu_\infty})\setminus\mu_{\infty_F}\subseteq \mathscr{M}_{2, \overline{F}}$ and $\operatorname{Eq}( \kappa_{F}, 0_{\mu_\infty})\setminus\mu_{\infty_F}= \mathscr{M}_{2, \overline{F}}$ as wanted.  
\end{proof}

\begin{corollary} \label{pcase}
The set $\mathscr{M}_{2,\overline{F}}^p$ of roots of unity $\zeta$ in $\overline{F}$ such that $\zeta$ is a $p$ power primitive root of unity for some prime number $p$ and integer $e$ and $[F(\zeta):F]=2$ can be expressed as the following equilizer:
\[
\mathscr{M}_{2,\overline{F}}^p = \operatorname{Eq}( \kappa_{F}^p, 0_{\mu_\infty}) \setminus \mu_{p^\infty_F}.
\]
\end{corollary}

\subsection{The constants \texorpdfstring{$\nu_{p^\infty_F}$}{Lg} and their properties}
In this section, our aim is to introduce fundamental constants essential for comprehending quadratic cyclotomic extensions. These constants provide important information about the structure of the set of quadratic cyclotomic extensions. 
\begin{definition}\label{maxvp}
\begin{enumerate} 
\item We define $\nu_{p^\infty_F}^+$ to be
$$\left\{\begin{array}{clll} \operatorname{max} \{ k \in \mathbb{N} | \zeta_{\operatorname{t}_F(p^k)} +
	\zeta_{\operatorname{t}_F(p^k)}^{-1} \in F, \ \exists \zeta_{p^k} \in \mathcal{P}_{p^k} \} & \text{when it exists} \\ 
	\infty & \text{ otherwise.}  \end{array} \right.$$
\item 	We define $\nu_{p^\infty_F}$  to be
	$$\left\{ \begin{array}{clll} \nu_{p^\infty_F}^++1 & \text{when} \ p=2\ \text{ and $F$ has property $\mathcal{C}_2$};  \\ 
	\nu_{p^\infty_F}^+ & \text{ otherwise. }  \end{array} \right.$$ 
\item We define the map $\kappa_{p, F}^+$ defined by
$$ \begin{array}{clllll} \kappa_{p,F}^+:& \mu_{p^\infty} & \rightarrow & \frac{F(\mu_{p^\infty})}{F}   \\ 
	& \zeta_{p^e} & \mapsto & [\zeta_{\operatorname{t}_F(p^e)} + \zeta_{\operatorname{t}_F(p^e)}^{-1}]_F
\end{array} $$
\end{enumerate} 
\end{definition}
\begin{remark}\label{remark-vpgeneral}
	Suppose $\nu_{2^\infty_F}<\infty$. We have $F$ has property $\mathcal{C}_2$ if and only if $\zeta_{\operatorname{t}_F(2^{\nu_{2^\infty_F}})}-\zeta_{\operatorname{t}_F(2^{\nu_{2^\infty_F}})}^{-1}\in F$. Moreover, $\nu_{2^\infty_F}=c_2$. In addition to that, $F$ does not have property $\mathcal{C}_2$  if and only if $\zeta_{\operatorname{t}_F(2^{\nu_{2^\infty_F}})}+\zeta_{\operatorname{t}_F(2^{\nu_{2^\infty_F}})}^{-1}\in F$. 
The above statements are direct consequences of Lemma \ref{vp-propertyQ} and Lemma \ref{prime-equaliser}.
\end{remark}
Let's begin by describing $\operatorname{Eq}(\kappa_{p,F}^+, 0_{\mu_\infty})$ as a group of roots of unity, where the orders of these roots of unity are entirely determined by the newly introduced constants when $p$ is odd. This result underscores the pivotal role of these constants in comprehending quadratic cyclotomic extensions.
\begin{lemma}\label{prime-equaliser}
	Let $p$ be a prime number. We have 
	$\operatorname{Eq}(\kappa_{p,F}^+, 0_{\mu_{p^\infty}})= \mu_{p^{
			\nu_{p^\infty_F}}}.$  
\end{lemma}
\begin{proof}
	We set ${\bf k} :=\nu_{p^\infty_F}$.  We want to prove that $\operatorname{Eq}( \kappa_{p,F}^+, 0_{\mu_\infty})= \mu_{p^{
			\nu_{p^\infty_F}}}$.   
	
	We will prove this in two cases:
	
	\textbf{Case 1}: We assume that ${\bf k} <\infty$.
	\begin{enumerate}
	\item We suppose that $p$ is odd or, $p=2$ and $F$ does not have property $\mathcal{C}_2$. 
	Let $\zeta_{p^{\bf k}}$ be a arbitrary primitive $(p^{\bf k})^{\operatorname{th}}$ root of unity in $\mu_{p^{\bf k}}$. By Definition \ref{maxvp}, $\nu_{p^\infty_F}^+=\nu_{p^\infty_F}$ and $\zeta_{\operatorname{t}_F(p^{\bf k})}+ \zeta_{\operatorname{t}_F(p^{\bf k})}^{-1}\in F$ proving that $\zeta_{p^{\bf k}}\in \operatorname{Eq}(\kappa_{p,F}^+, 0_{\mu_{p^\infty}})$. Using Lemma \ref{waring}, one can prove that we have that $\mathcal{P}_{p^{\bf k}}\subseteq \operatorname{Eq}(\kappa_{p,F}^+, 0_{\mu_{p^\infty}})$. 
	Now we take an arbitrary element of $\mu_{p^{\bf k}}$, that is $\zeta_{p^s}$ where $s\leq {\bf k}$, again using Remark \ref{tnf-remark} and Lemma \ref{waring}, we obtain $\zeta_{p^s}\in \operatorname{Eq}(\kappa_{p, F}^+, 0_{\mu_{p^\infty}})$.
	We prove the reverse inclusion, using Lemma \ref{waring} and noting, by Definition \ref{maxvp}, that ${\bf k}$ is the maximum number such that $\zeta_{p^{\bf k}}\in \operatorname{Eq}(\kappa_{p, F}^+, 0_{\mu_{p^\infty}})$.
	Therefore, we have proven that $\operatorname{Eq}(\kappa_{p, F}^+, 0_{\mu_{p^\infty}})=\mu_{p^{{\bf k}}}$. 
	\item We suppose that $p=2$ and $F$ has property $\mathcal{C}_2$.  By Remark \ref{remark-vpgeneral}, we obtain that $c_2=\nu_{2^\infty_F}$ . Moreover, for $e< c_2$, by Lemma \ref{property-q}, we deduce that
	$\zeta_{\operatorname{t}_F(2^e)}+\zeta_{\operatorname{t}_F(2^e)}^{-1}\in F$. From there, applying a reasoning as in $(1)$, we can again deduce the equality wanted.
	\end{enumerate}
	\textbf{Case 2:} When ${\bf k}=\infty$, the result follows directly from Definition \ref{zero-map} and Definition \ref{maxvp}. 
\end{proof}
\begin{remark}
Given $p$ be a prime number. We have 
	$\operatorname{Eq}(\kappa_{p,F}^+, 0_{\mu_{p^\infty}})= \mu_{p^{
			\nu_{p^\infty_F}^+}}.$  
\end{remark}
We now relate the constant $\nu_{2^\infty_F}^+$ with the constant $c_2$.
\begin{lemma}\label{vp-propertyQ}
	If $F$ has property $\mathcal{C}_2$, then $\nu_{2^\infty_F}^+=c_2-1<\infty$. 
\end{lemma}
\begin{proof}
This result is a direct consequence of Lemma \ref{property-q}.
\end{proof}

The following Theorem provides a valuable description of $\nu_{2^\infty_F}$ and it can be obtained by combining Theorem \ref{Equali}, Lemma \ref{prime-equaliser} and Theorem \ref{g2f}.

\begin{theorem}\label{nu2f}
	We have $$\nu_{2^\infty_F}=\left\{ \begin{array}{llll}\operatorname{max}\{k\in \mathbb{N}| F(\zeta_{2^k})=F(\zeta_4)\}, & \text{when it exists and }\ \zeta_4 \notin F;  \\ 
		\ell_{2^\infty_F}+1, & \text{when} \ \zeta_4 \in F \ \text{and}\ \ell_{2^\infty_F}<\infty;\\
		\infty , & \text{otherwise.}  \end{array} \right.$$ 
We note that $\nu_{2^\infty_F}=\infty$ if and only if $\ell_{2^\infty_F}=\infty$ or $\ell_{2^\infty_F}=1$, and for all $k\in \mathbb{N}$, $F(\zeta_{2^k})=F(\zeta_4)$.
\end{theorem}

The following result can be derived from the preceding theorem.
 \begin{corollary}\label{equaltozeta4}
	If $\zeta_4 \notin F$, then $F(\zeta_4)=F(\zeta_{2^i})$ for all $i\in \{2, \cdots, \nu_{2^\infty_F}\}$.
\end{corollary}

\subsection{Describing quadratic \texorpdfstring{$p$}{Lg}-cyclotomic sets as difference of two groups}
We are ready to describe the set $\mathscr{M}_{2, \overline{F}}^p$ of roots of unity in $\mu_{p^{\infty}}$ defining quadratic extensions and the set $\mathscr{C}_{2,\overline{F}}^p$ of quadratic $p$ power cyclotomic extension, thanks to the constant $\nu_{p^\infty_F}$. Combining Lemma \ref{prime-equaliser} and Theorem \ref{Equali}, we obtain easily the following theorem.
\begin{theorem}\label{setofgenerators}
Let $p$ be a prime number. We have the following:
\begin{enumerate}
	\item The set of roots of unity in $\mu_{p^{\infty}}$ defining quadratic cyclotomic extensions can be described as a difference of two groups:
	\[
	\mathscr{M}_{2, \overline{F}}^p=\mu_{p^{\nu_{p^\infty_F}}}\setminus \mu_{p_F^\infty}.
	\]
		
	\item The set of quadratic $p$ power cyclotomic extensions can be expressed as:
	\[
	\mathscr{C}_{2,\overline{F}}^p = \begin{cases}
		\emptyset, & \text{when } \ell_{p^\infty_F}=\nu_{p^\infty_F}; \\
		\left\{  F\left(\zeta_{p^{\ell_{p^\infty_F}+1}} \right) \right\}, & \text{otherwise}.
	\end{cases}
	\]
\end{enumerate}
\end{theorem}

\begin{remark}
	\begin{enumerate} 
		\item $\mathscr{M}_{2,\overline{F}}^{2}=\emptyset$ if and only if $\mu_{2^{\infty}_F}=\mu_{2^{\infty}}$. However, when $p$ is odd, $\mathscr{M}_{2,\overline{F}}^{p}$ can be equal to $\emptyset$ and $ \mu_{p^{\infty}}\neq \mu_{p^{\infty}_F}$ (see Corollary \ref{degree-odd}). 
		\item When $\mathscr{M}_{2,\overline{F}}^{p}\neq \emptyset$, then $\ell_{p^\infty_F} \neq \nu_{p^\infty_F}$. We note that, when $\nu_{p^\infty_F}>0$, $F(\zeta_{p^{\nu_{p^\infty_F}}})=F(\zeta_{p})$. 
		\item When $\ell_{2^\infty_F}<\infty$ and $\ell_{2^\infty_F}\neq 1$ we have $\mathscr{M}_{2, \overline{F}}^{2}= \mathscr{G}_{2,\overline{F}}^2=\mathcal{P}_{2^{\ell_{2^\infty_F}+1}}$ (see Theorem \ref{g2f}) and therefore, $\mathscr{C}_{2, F}^2 = \mathscr{R}_{2,\overline{F}}^2=  \mathscr{R}_{2,\overline{F}}$.

	\end{enumerate}
\end{remark}

\subsection{Describing the set of quadratic cyclotomic primitive roots of unity using groups} 

In the above we have described the set of quadratic cyclotomic extensions generated by primitive $p$ power primitive roots of unity. We now combine those results to describe the set of quadratic cyclotomic extensions in general. We start by the following definitions that will be useful in constructing the required results.
\begin{definition}\label{sn}
\begin{enumerate}
\item Let $n\in \mathbb{N}$. We define 
$$S_{F,n}=\{p\in \mathbb{P}| \ p|\operatorname{o}_F(\zeta_n)\}.$$
\item We define the set 
$$\mathcal{S}_F:=\{S\in P(\mathbb{P})| 
\forall \ p\in S, \exists e_p \in \mathbb{N}:[ \zeta_{p^{e_p}}\in \mathscr{M}_{2, \overline{F}}\ \wedge \forall \ B\finsub S, \zeta_{\prod\limits_{p\in B}p^{e_p}} \in \mathscr{M}_{2, \overline{F}}]\}.$$
\end{enumerate} 
\end{definition}
Any chain in $\mathcal{S}_F$ admits an upper bound. The proof is omitted as it is straightforward.
\begin{lemma} 
For a given maximal chain $\mathbf{C}:=S_1\subseteq S_2\subseteq \cdots \subseteq S_n$ in $\mathcal{S}_F$, the set $ \bigcup \limits^{n}_{i=1}S_i$ is an upper bound for the chain $\mathbf{C}$. 
\end{lemma}

The preceding lemma allows us to define the collection of upper bounds for maximal chains within $\mathcal{S}_F$, and for each element within this collection, we can associate canonically a group.
\begin{definition} \label{munF}
    Let $\mathcal{S}_{F,\operatorname{max}}$ be defined as the set of upper bounds of maximal chains in $\mathcal{S}_F$. 
\end{definition}
In the following Lemma, we prove that two elements of $\mathcal{S}_{F, \operatorname{max}}$ are either disjoint or equal. 
\begin{lemma}\label{powerset-cyclotomic}
For any $M_1, M_2 \in \mathcal{S}_{F,\operatorname{max}}$, $M_1\cap M_2 \neq \emptyset $ if and only if $M_1 = M_2$. 
\end{lemma}

\begin{proof}
Assume $M_1, M_2 \in \mathcal{S}_{F,\operatorname{max}}$. Suppose $M_1\cap M_2 \neq \emptyset $, then there exists $p_0 \in M_1 \cap M_2$, implying $\zeta_{p_0^{e_{p_0}}}\in \mathscr{M}_{2, \overline{F}}$ for some $e_{p_0} \in \mathbb{N}$.

Suppose, for contradiction, that $M_1 \neq M_2$. Without loss of generality, assume $q \in M_1\backslash M_2$, where $q$ is a prime. We show that $M_2 \cup \{q\}\in \mathcal{S}_{F}$. Let $B \subseteq M_2 \cup \{q\}$. If $B \subseteq M_2$, then $\zeta_{\prod\limits_{p\in B}p^{e_p}} \in \mathscr{M}_{2, \overline{F}}$, for some $e_p \in \mathbb{N}$, where $p \in B$. Now, if $q \in B$, $\zeta_{q^{e_q}p_0^{e_{p_0}}} \in \mathscr{M}_{2, \overline{F}}$, and  $\zeta_{p_0^{e_{p_0}} \prod_{p\in B}^s p^{e_p}} \in \mathscr{M}_{2, \overline{F}}$, for some $e_q$, $e_{p_0}$ and $e_p\in \mathbb{N}$ where $p\in B$.
Therefore, by Lemma \ref{equality}, we have $\zeta_{q^{e_q}p^{e_{p_0}} \prod_{p\in B} p^{e_p}} \in \mathscr{M}_{2, \overline{F}}$. This contradicts the maximality of $M_2$ and, therefore, $M_1 = M_2$. The converse is trivial and the lemma is proven.
\end{proof}
In the following two results, we establish a connection between $\mathscr{M}_{2, \overline{F}}$ and $\mathcal{S}_{F,\operatorname{max}}$. 

\begin{lemma}\label{snn} 
Let $\zeta_n\in \mathscr{M}_{2, \overline{F}}$. Then, $S_{F,n} \neq \emptyset$, and there exists a unique $M\in \mathcal{S}_{F,\operatorname{max}}$ such that $S_{F,n}\subseteq M$. 
\end{lemma}

\begin{proof}
Suppose $\zeta_n\in \mathscr{M}_{2, \overline{F}}$. That is, $[F(\zeta_n):F]=2$. Also, $S_{F,n}\neq \emptyset$ as $\zeta_n\notin F$ by Definition \ref{sn}. Furthermore, for each $p\in S_{F,n}$, we have $\zeta_{p^{\text{\textepsilon}_n(p)}}\notin F$. This implies $[F(\zeta_{p^{\text{\textepsilon}_n(p)}}):F]=[F(\zeta_n):F]=2$. Consequently, $S_{F,n}\in \mathcal{S}_{F}$ according to Definition \ref{sn}. Therefore, there exists a unique $M\in \mathcal{S}_{F,\operatorname{max}}$ such that $S_{F,n}\subseteq M$, and the uniqueness of $M$ follows from Lemma \ref{powerset-cyclotomic}. 
\end{proof}

The following lemma establishes a connection between $\mathcal{S}_{F,\operatorname{max}}$ and the equality of two cyclotomic fields. It follows easily from Lemma \ref{equality} and therefore the proof is omitted.
\begin{lemma}\label{cyclotomic-powersetofprimes}
We have $\zeta, \zeta'\in\mathscr{M}_{2, \overline{F}}$ and $F(\zeta)=F(\zeta')$ if and only if there exists $M\in \mathcal{S}_{F,\operatorname{max}}$ such that $\zeta , \zeta'\in \bigvee \limits_{p\in M}\mu_{p^{\nu_{p^\infty_F}}} \bigvee \limits_{p \in \mathbb{P}\setminus M}\mu_{p^{\ell_{p_F^\infty}}}\setminus\mu_{\infty_F}$.
\end{lemma}

We arrive at the main theorem of this section, describing the set of roots of unity defining quadratic extensions as the difference of two groups. We observe again that the constants $\nu_{p,F}$ and $\ell_{p_F^\infty}$ become fundamental constants in the study of the set of roots of unity defining quadratic extensions. Additionally, we identify $\mathcal{S}_{F,\operatorname{max}}$ with the set of quadratic $p$ power cyclotomic extensions.

\begin{introthm1}{Theorem} \label{m2F}
	\begin{enumerate}
		\item The set of roots of unity in $\mu_{\infty}$ defining quadratic extensions can be described as a difference of two groups: $$\mathscr{M}_{2, \overline{F}}=\coprod \limits_{M\in \mathcal{S}_{F,\operatorname{max}}}(\bigvee \limits_{p\in M}\mu_{p^{\nu_{p^\infty_F}}} \bigvee \limits_{p \in \mathbb{P}\setminus M}\mu_{p^{\ell_{p_F^\infty}}}\setminus\bigvee \limits_{p \in \mathbb{P}}\mu_{p^{\ell_{p_F^\infty}}}),$$
		\item The set of quadratic $p$ power cyclotomic extensions can be expressed as: $$\mathscr{C}_{2,\overline{F}}\simeq \mathcal{S}_{F,\operatorname{max}}.$$
	\end{enumerate}
\end{introthm1}
\begin{proof}	
 \begin{enumerate}
 	\item This statement follows directly from Theorem \ref{setofgenerators}, Lemma \ref{snn} and Lemma \ref{cyclotomic-powersetofprimes}.
 	\item We denote $\mu_M:= \bigvee \limits_{p\in M}\mu_{p^{\nu_{p^\infty_F}}} \bigvee \limits_{p \in \mathbb{P}\setminus M}\mu_{p^{\ell_{p_F^\infty}}}$. We want to prove that $\mathscr{C}_{2,\overline{F}}\simeq \mathcal{S}_{F,\operatorname{max}}$. Consider the map $$\begin{array}{llll} \Psi :&\mathcal{S}_{F,\operatorname{max}}& \rightarrow &\mathscr{C}_{2,\overline{F}}\\ 
 		&M & \mapsto &  F(\zeta_n)\end{array}$$ where $\zeta_n \in \mu_M\setminus\mu_{\infty_F}$. Indeed, $\Psi$ is well defined since all elements in $\mathcal{S}_{F,\operatorname{max}}$ are distinct by Lemma \ref{powerset-cyclotomic} and it is clear that, for any $M\in \mathcal{S}_{F,\operatorname{max}}$, all elements in $\mu_M\setminus\mu_{\infty_F}$ define a unique extension in $\mathscr{C}_{2,\overline{F}}$. 
		
 	Now we prove that $\Psi$ is injective.  Let $F(\zeta_n), F(\zeta_m)\in \mathscr{C}_{2,\overline{F}}$ where $\zeta_n\in \mu_M\setminus\mu_{\infty_F}$, and  $\zeta_m \in \mu_N\setminus\mu_{\infty_F}$.
 Suppose that $F(\zeta_n)=F(\zeta_m)$. Then by Lemma \ref{cyclotomic-powersetofprimes}, there is $L\in \mathcal{S}_{F,\operatorname{max}}$ such that  $\zeta_n, \zeta_m\in \mu_L\setminus\mu_{\infty_F}$. This shows that $M\cap L\neq \emptyset$ and $N\cap L\neq \emptyset$. As a result, $M= L = N$, by Lemma \ref{powerset-cyclotomic}. Hence, $\Psi$ is injective. 
 	 The surjectivity follows since for every $F(\zeta_n)\in \mathscr{C}_{2,\overline{F}}$ we have  $M\in \mathcal{S}_{F,\operatorname{max}}$ such that $S_{F,n} \subseteq M$ by Lemma \ref{snn}. Therefore, $\zeta_n \in \mu_M\setminus\mu_{\infty_F}$.  Therefore, $\Psi$ is bijective. 
 \end{enumerate}
\end{proof}

\subsection{Quadratic cyclotomic extensions over finite fields}
In the following lemma, we provide a treatment of quadratic cyclotomic extensions over finite fields according to the findings of this paper.
\begin{lemma}\label{finitequadratic}
	Let $p$ be a prime number with $p \neq \wp$. Let $q = \wp^m$ for some $m \in \mathbb{N}$. 
	\begin{enumerate}
		\item When $p$ is an odd prime number, then $p^e \mid (q + 1)$ if and only if $\mathbb{F}_{q^2} = \mathbb{F}_{q}(\zeta_{p^j})$ for all $j \in \{1, \cdots, e\}$.
		
		\item  We have $\mathscr{M}_{2,\overline{\mathbb{F}_q}}^p = \mu_{p^{\text{\textepsilon}_{q^2-1}(p)}}\setminus \mu_{p^{\text{\textepsilon}_{q-1}(p)}} $.
		\item  We have $\mathscr{G}_{2, \overline{\mathbb{F}_q}}=\mathcal{P}_{2^{\text{\textepsilon}_{q^2-1}(2)+1}}\bigvee  \mu_m $ where $m=\operatorname{q}_{q-1}(2)$.
	\end{enumerate}
	In particular, we have 
	\begin{itemize}
		\item $\nu_{p^\infty_{\mathbb{F}_q}}=\text{\textepsilon}_{q^2-1}(p)$ and $\ell_{p^\infty_{\mathbb{F}_q}}=\text{\textepsilon}_{q-1}(p)$
		\item $\operatorname{d}_{\mathbb{F}_q}(n)=(n, q-1)$ where $n\in \mathbb{N}$.
		\item When $p$ is odd, $\text{\textepsilon}_{q^2-1}(p)=\text{\textepsilon}_{q+1}(p)$
		\item When $p=2$, 
		\begin{itemize}
			\item $\text{\textepsilon}_{q-1}(2)>1,$ then $\text{\textepsilon}_{q^2-1}(2)=\text{\textepsilon}_{q-1}(2)+1$ and $\mathscr{M}_{2,\overline{\mathbb{F}_q}}=\mathscr{G}_{2,\overline{\mathbb{F}_q}}$
			\item $ \text{\textepsilon}_{q-1}(2)=1$, then $\text{\textepsilon}_{q^2-1}(2)=\text{\textepsilon}_{q+1}(2)+1$ and $\mathbb{F}_q$ has property $\mathcal{C}_2$. 
		\end{itemize}
	\end{itemize}
\end{lemma}
\begin{proof}
	\begin{enumerate}
		\item Suppose that $p$ is an odd prime number such that	$p^e \mid (q + 1)$. Then, we have $p \mid (q + 1)$, so that $p \mid (q^2-1)=(q-1)(q+1)$. Moreover, we have $p \nmid (q-1)$. Indeed, assuming the opposite, i.e., $p \mid (q-1)$, would lead to $p \mid (q+1)-(q-1)=2$, which is impossible since $p$ is odd. Therefore, $\mathbb{F}_{q^2} = \mathbb{F}_{q}(\zeta_{p^j})=\mathbb{F}_{q}(\zeta_p)$, for any $j \in \{ 1, \cdots, e\}$. The converse follows from the fact that $\zeta_{p^e} \in {\mathbb{F}_{q^2}}$ if and only if $p^e \mid (q^2-1)$ and $p^e \nmid (q-1)$.
		\item From (1), we deduce that $\nu_{p^\infty_{\mathbb{F}_q}}=\text{\textepsilon}_{q^2-1}(p)$. Moreover, we have $\zeta_{p^e} \in {\mathbb{F}_q}$ implies $p^e \mid (q-1)$. So that $\ell_{p^\infty_{\mathbb{F}_q}}=\text{\textepsilon}_{q-1}(p)$.
		\item Follows from (2) above and Theorem \ref{g2f}. Indeed, $\zeta_m \in {\mu_{2\infty+1}}_{\mathbb{F}_q}$ implies $m$ is odd by Definition \ref{mu_infinity} and $m \mid (q-1)$. Moreover, the maximal $m$ having this property is $m=\operatorname{q}_{q-1}(2)$.
	\end{enumerate}
	\begin{itemize}
		\item This is proven above. 
		\item Clear by definition of $\operatorname{d}_F(n)$. 
		\item When $p$ is odd, $\nu_{p^\infty_{\mathbb{F}_q}}=\text{\textepsilon}_{q+1}(p)$ by (1).
		
		\item When $p=2$,
		\begin{itemize}
			\item If $\text{\textepsilon}_{q-1}(2)>1$, then $2$ divides both $q+1$ and $q-1$ since $q$ is odd by the initial assumption. Also, we have $(q+1)-(q-1)=2$. That implies $\text{\textepsilon}_{q+1}(2)=1$ and $\text{\textepsilon}_{q^2-1}(2)=\text{\textepsilon}_{q-1}(2)+1$.	
			\item If $\text{\textepsilon}_{q-1}(2)=1$, then $\text{\textepsilon}_{q^2-1}(2)=\text{\textepsilon}_{q+1}(2)+1$ and $\mathbb{F}_q$ has property $\mathcal{C}_2$. 
		\end{itemize}
	\end{itemize}
\end{proof}

\section{The structure of the sets of quadratic extensions}\label{moduli-quadratic}

To view cyclotomic field extensions of degree $2$ within the full set of quadratic extensions, we include succinctly the treatment of the sets of quadratic extensions. Further, we will inject the sets of quadratic cyclotomic extensions into the sets of the general quadratic extensions.  

\subsection{Embedding the set of quadratic cyclotomic extensions into the set of separable quadratic extensions}
In this section, $s({F}) = \{ a \in F| \exists {b \in F}: [a =b^2]\}$. We note that $s({F})^*$ is a mutiplicative subgroup of $F^*$. The following lemma determines the structure of the set of quadratic extensions over a field of characteristic not $2$. This structure is simply the translation of Kummer's theory (see \cite[Theorem 5.8.5, Proposition 5.8.7]{VS}). 

\begin{lemma}\label{moduliquadratic}
	Let $F$ be a field of characteristic not $2$, ${\mathcal{Q}_{2,F}}_{\operatorname{iso}}$ be a set of quadratic extensions over $F$ up to isomorphism. Then $$  {\mathcal{Q}_{2,F}}_{\operatorname{iso}}\cong \frac{F^*}{s({F})^*}\setminus \{ s({F})^*\}.$$
\end{lemma}

\begin{proof}
To prove the result we can simply prove that the map $\varphi:\frac{F^*}{s({F})^*}\setminus\{ s({F})^*\} \rightarrow  {\mathcal{Q}_{2,F}}_{\operatorname{iso}}$ defined by $\varphi(as({F})^*)=\left[F[x]/ \left\langle x^2-a\right\rangle   \right]_{\operatorname{iso}}$ is well-defined and bijective.
\end{proof}

Over the field of characteristics different from $2$, all quadratic extensions are simple radical. If $F(\zeta_n)/F$ is a quadratic we know by (1) Corollary \ref{generator} that $\zeta_n -\zeta_n^{\text{ \textlyoghlig}_n}$ is a radical generator for $F(\zeta_n)$ over $F$. To embed the structure of the set of quadratic cyclotomic extensions along with the set of all quadratic extensions over fields of characteristic not equal to $2$, we define the following map.

\begin{definition}\label{chi-rad}
	Let $n \in \mathbb{N}$. When $\wp \neq 2$, we define a map $$\begin{array}{ccll}\text{\texthtrtaild}_{\operatorname{rad}} :&\mathscr{M}_{2,\overline{F}}& \rightarrow & \frac{F^{*}}{s(F)^{*}}\setminus\left\{ s(F)^*\right\}\\ 
		&\zeta_n& \mapsto & [ \zeta_n^2+\zeta_n^{2j_n}-2\zeta_n^{j_n+1}]_{s({F})^{*}}. \end{array}$$
\label{injective-cyclotomic}
	We have a bijective map 
	$$\mathscr{C}_{2,\overline{F}} \simeq \text{\texthtrtaild}_{\operatorname{rad}} (\mathscr{M}_{2,\overline{F}}).$$ Moreover, this bijection induces an injective map $\mathscr{C}_{2,\overline{F}}  \hookrightarrow {\mathcal{Q}_{2,F}}_{\operatorname{iso}}$.
\end{definition}

For quadratic extensions over a field of characteristic $2$, the structure is just the translation of Artin-Schreier's theory  (see \cite[Theorem 5.8.4, Proposition 5.8.6]{VS}). 

\begin{lemma}
	Let $F$ be a field of characteristic $2$, ${\mathcal{Q}_{2,F}^{\operatorname{sep}}}_{\operatorname{iso}}$ be a set of separable quadratic extensions of $F$ up to isomorphism. Then $$ {\mathcal{Q}_{2,F}^{\operatorname{sep}}}_{\operatorname{iso}} \simeq \frac{F}{\mathcal{A}(F)}\setminus  \{\mathcal{A}(F)\}$$ where $\mathcal{A}(F)=\{a^2-a|a\in F\}$ is an additive subgroup of $F$.
\end{lemma}

\begin{proof} 
To prove the result we can simply prove that the map $\varphi:\frac{F}{\mathcal{A}(F)}\setminus \{\mathcal{A}(F)\}\rightarrow {\mathcal{Q}_{2,F}^{\operatorname{sep}}}_{\operatorname{iso}}$ defined by $\varphi(a+\mathcal{A}(F))=\left[F[x]/ \left\langle x^2-x-a\right\rangle \right]_{\operatorname{iso}},$ is well-defined and bijective. 
\end{proof}

Over the field of characteristic $2$, all quadratic extensions are Artin-Schreier. If $F(\zeta_n)/F$ is a quadratic we know by (2) Corollary \ref{generator} that $\frac{\zeta_n }{\zeta_n +\zeta_n^{\text{ \textlyoghlig}_n}}$ is an Artin-Schreier generator for $F(\zeta_n)$ over $F$. To embed the structure of the set of quadratic cyclotomic extensions along within the set of all separable quadratic extensions over fields of characteristic $2$, we define the following embedding.

\begin{definition}\label{chi-rad2} Let $n \in \mathbb{N}$. When $\wp= 2$, we define a map $$\begin{array}{ccll}\text{\texthtrtaild}_{\operatorname{as}} :&\mathscr{M}_{2,\overline{F}}& \rightarrow & \frac{F}{\mathcal{A}(F)}\setminus \{\mathcal{A}(F)\}\\ 
		&\zeta_n& \mapsto & \left[ \frac{\zeta_n^{\text{ \textlyoghlig}_n+1} }{ \zeta_n^2+\zeta_n^{2\text{ \textlyoghlig}_n}+2\zeta_n^{\text{ \textlyoghlig}_n+1}} \right]_{\mathcal{A}(F)}. \end{array}$$

	We have a bijective map 
	$$ \mathscr{C}_{2,\overline{F}} \simeq \text{\texthtrtaild}_{\operatorname{as}}(\mathscr{M}_{2,\overline{F}})$$ This map induces an injective map $\mathscr{C}_{2,\overline{F}}  \hookrightarrow {\mathcal{Q}_{2,F}^{\operatorname{sep}}}_{\operatorname{iso}}$. 	
\end{definition}

\subsection{The structure of the set of inseparable quadratic extensions}
We include a description of the structure of the set of the inseparable extensions of degree $2$ over the field of characteristic $2$, for completeness. Since any element not in the base field serves as an inseparable generator for an inseparable extension of degree $2$, one can easily prove the following lemma.
\begin{lemma}\label{isomorphic-inseparable}
	Let $F$ be a field of characteristic  $2$ and let $K$ (resp. $L$) be a purely inseparable quadratic extension of $F$ defined by the minimal polynomial  $x^2-a$ (resp. $x^2-a')$. Then $K$ and $L$ are isomorphic over $F$ if and only if $a=c^2a'-b^2$ for some $b,c\in F$.
\end{lemma}

Based on the relation derived from the system of parameters in the preceding lemma, we extract a group action that enables us to characterize the isomorphism classes as an orbit under this action.
\begin{definition}
	Let $F$ be a field of characteristic $2$. 
	\begin{enumerate} 
		\item We define $\varphi : s({F})^* \rightarrow \text{Aut}(s(F))$ as the map such that $\varphi(c) = \varphi_{c}$ where $c \in s({F})^*$ and $\varphi_{c} : s(F) \rightarrow s(F)$ is defined as $\varphi_{c}(b) = cb$ for some $b \in s(F)$. 
		\item We define the semi-direct product $s({F})^* \ltimes_{\varphi} s(F)$ to be the set $s({F})^* \times s(F) $ endowed with the binary operation $(c,b)(c',b') = (cc', \varphi_{c'}(b) + b')$. One can prove that $s({F})^* \ltimes_{\varphi} s(F)$ is a group.  
		\item We define a group action of $F$ on $s({F})^* \ltimes_{\varphi} s(F)$ denoted by $\rho : F \times s({F})^* \ltimes_{\varphi} s(F) \rightarrow F$ as the map $(a, (c,b)) \mapsto \rho(a, (c,b)) = ca - b$. 
		We denote $O_{\rho}(a)$ to be the orbit of $a \in F$ with respect to $\rho$ and $\frac{F^*}{ s({F})^* \ltimes_\varphi s(F)}$ the set of these orbits.
	\end{enumerate}
\end{definition}

\begin{lemma}
	Let $F$ be a field of characteristic $2$ , ${\mathcal{Q}_{2,F}^{\operatorname{insep}}}_{\operatorname{iso}}$ be a set of inseparable quadratic extensions up to isomorphism. Then $${\mathcal{Q}_{2,F}^{\operatorname{insep}}}_{\operatorname{iso}} \simeq \frac{F^*}{ s({F})^* \ltimes_\varphi s(F)} \setminus \{ O_\rho(1) \}$$
\end{lemma}
\begin{proof}
To prove the result we can simply prove that the map $$\varphi:\frac{F^*}{s({F})^* \ltimes_\varphi s(F)}\setminus \{ O_\rho(1)\} \rightarrow \mathcal{Q}_{2,F}^{\operatorname{insep}}$$ be defined as $\varphi(O_\rho(a))=[F[x]/\langle x^2-a \rangle]_{\operatorname{iso}}$ is well-defined and bijective. 
\end{proof}

\bibliographystyle{Abbrv}
	
\end{document}